\documentclass[10pt,twocolumn,reqno]{amsart}

\usepackage{times}
\usepackage{epsfig}
\usepackage{graphicx}
\usepackage{amsmath}
\usepackage{amssymb}
\usepackage[margin=2cm]{geometry}

\usepackage{enumerate}
\usepackage{amsthm}
\usepackage{xypic}
\allowdisplaybreaks

\newtheorem{proposition}{Proposition}

\newtheorem{lemma}{Lemma}

\usepackage[pagebackref=true,breaklinks=true,colorlinks,bookmarks=false]{hyperref}

\begin{document}

\title{The Square Root Velocity Framework for Curves in a Homogeneous Space}


\author{Zhe Su,\quad Eric Klassen,\quad Martin Bauer}
\address{Florida State University, Department of Mathematics, Tallahassee, FL}
\email{zsu@math.fsu.edu, klassen@math.fsu.edu, bauer@math.fsu.edu}
\thanks{Eric Klassen gratefully acknowledges the support of the Simons Foundation (Grant \# 317865).}

\begin{abstract}
In this paper we study the shape space of  curves with values in a 
homogeneous space $M = G/K$, 
where $G$ is a Lie group and $K$ is a compact Lie subgroup. We generalize the square root velocity 
framework to obtain a reparametrization invariant metric on the space of curves in $M$. 
By identifying curves in $M$ with their horizontal lifts in $G$, geodesics then can be 
computed.

We can also mod out by reparametrizations and by rigid motions of $M$. 
In each of these quotient spaces, we can compute Karcher means, geodesics, and perform 
principal component analysis. We present numerical examples including the analysis of a set of hurricane paths.
\end{abstract}
\maketitle

\section{Introduction}

The field of shape analysis is concerned with the mathematical description, comparison and analysis of geometric shapes. 
This has applications in a variety of fields in pure and applied mathematics. Examples include computational anatomy,  
medical imaging, computer vision and functional data analysis. Although the notion of  \emph{shape} varies widely depending on the specific application, many shape spaces share two common difficulties:  these spaces are usually non-linear and often infinite-dimensional. To deal with the resulting challenges the methods of (infinite-dimensional) Riemannian geometry have proved to be a successful approach. 

The notion of shape space that we adopt for the purpose of this article is the space of unparametrized curves with values in a homogeneous space $M$. Before we describe the contributions of the current work, we want to summarize previous work in this area.  We start by considering the case of smooth, open, regular curves with values in some Euclidean space $\mathbb R^d$:
\begin{equation}
	\operatorname{Imm}([0,1],\mathbb R^d):=\left\{c\in C^{\infty}([0,1],\mathbb R^d): c'\neq 0 \right\}.
\end{equation}
In the field of shape and functional data analysis one is usually not interested in the actual parametrization of the curves, but only in their \emph{geometric shape}. Mathematically the space of unparametrized curves (shapes) can be modeled as the quotient space
\begin{equation}
	\mathcal S([0,1],\mathbb R^d)=\operatorname{Imm}([0,1],\mathbb R^d)/\operatorname{Diff}^+([0,1])\;,
\end{equation} 
where $\operatorname{Diff}^+([0,1])$ denotes the group of smooth orientation preserving diffeomorphisms of the interval $[0,1]$ onto itself. 

For applications in shape analysis we want to define a distance or similarity measure on the space of un-parametrized curves. 
Towards this aim one can equip the space of all para\-metrized curves with a $\operatorname{Diff}^+([0,1])$-invariant metric and 
induce a Riemannian metric on the quotient space by requiring the quotient map $\pi$ to be a Riemannian submersion.  
By an invariant Riemannian metric on $\operatorname{Imm}([0,1],\mathbb R^d)$, we mean a Riemannian metric $G$ with the property that the re-para\-metriza\-tion group $\operatorname{Diff}^+([0,1])$ acts by 
isometries, i.e.,
\begin{equation}
	G_{c\circ\varphi}(u\circ\varphi, u\circ\varphi) = G_c(u,u)\,.\quad \forall \varphi \in \operatorname{Diff}^+([0,1])\;.
\end{equation}
Given an invariant metric $G$ on $\operatorname{Imm}([0,1],\mathbb R^d)$, the induced metric\footnote{The invariance property is a necessary but not a sufficient condition to induce a Riemannian metric on the quotient space. In the context of invariant metrics on spaces of curves, however, it has been shown that they indeed induce a smooth Riemannian metric on the quotient space.} on $\mathcal S([0,1],\mathbb R^d)$ is then defined as
\begin{equation}
	G_{\pi(c)}(u, u) = \inf_{\pi_{*c}(h) = u} G_c(h,h)\,.
\end{equation}
Thus the study of Riemannian metrics on the shape space $\mathcal S([0,1],\mathbb R^d)$ is reduced to the study of invariant metrics on the simpler space of parametrized curves. 

It came as a big surprise that the simplest such metric -- the reparametrization invariant $L^2$-metric -- induces vanishing geodesic distance, which renders it useless for applications in shape analysis, see \cite{MiMu2005,MiMu2006,BaBrHaMi2012}. To overcome this degeneracy several modifications of the $L^2$-metric have been introduced: Michor and Mumford \cite{MiMu2006} introduced metrics that are weighted by the curvature of the foot point curve $c$ and Shah \cite{Sh2008} 
and Mennucci and Yezzi \cite{YeMe2004} studied length weighted $L^2$-metrics. While overcoming the degeneracy of the geodesic distance, the existence of length minimizing curves for these metrics remains a delicate problem. It turned out to be a more promising approach to include (arc-length) derivatives of the tangent vector in the definition of the metric, yielding the class of Sobolev metrics. 
These metrics have received rigorous theoretical analysis and, in particular, there exist analytic results on local and global existence of geodesics \cite{MiMu2007,MeYeSu2008,BrMiMu2014}.

From an application point of view, a certain family of first order Sobolev metrics proved advantageous, as there exist isometric transformations to flat spaces allowing for explicit calculations of geodesics and geodesic distance \cite{MiSrJo2007,SrKlJoJe2011,SuMeSoYe2011}. This family of metrics, also called elastic metrics, can be written as:
\begin{equation}\label{elastic_metric}
G_c(h,h) =\int_0^1 a^2 |D_s h^N|^2+ b^2 | D_s h^T|^2 ds; 
\end{equation}
here $D_s$ and $ds$ denote differentiation and integration with respect to arc-length and let $D_s h^N$ (resp. $D_sh^T$) denote the components of $D_sh$ which are normal (resp. tangent) to the tangent vector $\dot c$ of the curve.
For $a=b$ and curves with values in $\mathbb R^2$, Younes et al. \cite{YoMiShMu2008} introduced the \emph{basic mapping} to represent this metric;  for the space of curves with values in general Euclidean space $\mathbb R^d$ Srivastava et al. \cite{SrKlJoJe2011} developed the SRV transform to represent the metric with $a=1$ and $b=\frac12$. These transformations have been generalized to arbitrary parameters $a$, $b$ in \cite{BaBrMaMi2014}. Using the SRV, efficient numerical calculations of geodesics have been developed and it also has given rise to rigorous results on the metric completion and the existence of minimizing reparametrizations \cite{BrMiMu2014}. In particular  it has been shown that the metric completion of $\operatorname{Imm}([0,1],\mathbb R^d)$ is the space of absolutely continuous functions $AC([0,1],\mathbb R^d)$ \cite{MeYeSu2008} and that in the case of PL curves \cite{LaRoKl2015}  and $C^1$ curves \cite{Bru2016}, optimal reparametrizations exist, leading to length-minimizing paths in shape space $\mathcal S([0,1],\mathbb R^d)$. 

Recently, there has been an effort to generalize these metrics (and in particular the SRV transform) for curves with values in a general Riemannian manifold $M$. Su et al. \cite{SuKuKlSr2014} introduced the TSRVF (transported square root velocity function), in which all SRVFs are parallel transported along geodesics to the tangent space at a single reference point $x\in M$. This method is computationally effective, but it has the disadvantage of introducing distortions for curves that venture far away from $x$, and the metric depends on the chosen reference point $x$. Zhang et al. \cite{ZhSuKlLeSr2015} introduced a different adaptation of the SRV, in which each path $\alpha:[0,1]\to M$ is represented by a path in the tangent space at its own initial point $\alpha(0)$; the velocity vectors are parallel translated along the path $\alpha$ itself to this initial point. The paths are then compared using a metric on the total space of the tangent bundle $TM$. This method avoids the distortion and arbitrariness of the TSRVF resulting from the choice of a reference point; however, the computations are much more difficult. Le Brigant \cite{LeArBa2015} introduced a more intrinsic metric on curves, defined pointwise along the curve. This method also avoids the arbitrariness and distortion of the TSRVF, but at a greater computational cost. 

In  \cite{CeEsSch2015}, Celledoni et al. adapted the SRV framework to the analysis of curves in a Lie group with a right-invariant metric. The basic idea is to use right translation to identify all tangent vectors to elements of the Lie algebra. The approach taken in the current paper is a generalization of this idea to curves in homogenous spaces. 

For the space $\operatorname{Imm}(N,M)$ of immersions between two possibly higher dimensional manifolds much less is known. Sobolev metrics thereon have been introduced in \cite{BaPhMi2011} and in the case of surfaces in $\mathbb R^d$ certain generalizations of the SRV framework have been studied in \cite{SaKuSrCa2014,JeKuKlSr2012,KuKlDiSr2010,KuKlGoDiSr2012,KuKlDiJaJaAvSr2011,KuJeXiKlLa2015}. 
 For more details on Riemannian metrics on spaces of curves and surfaces we refer to \cite{Younes2010,SrKl2016,BaBrMi2014,BaBrMi2016}.

{\bf Contributions of this paper:} We introduce a new generalization of the SRV transform for curves with values in a homogeneous space $M=G/K$, where $G$ is a Lie group and $K$ is a compact Lie subgroup. Many of the Riemannian manifolds that arise in applications can be viewed as homogeneous spaces, for example Euclidean spaces, spheres, Grassmannians, hyperbolic spaces, positive definite symmetric matrices, as well as all Lie groups. Compared to previous attempts, our approach has the advantage that it still yields explicit formulas for geodesics and geodesic distance -- computing a geodesic on the space of parametrized curves is equivalent to (1) computing a geodesic in $G$ and (2) performing an optimization over the compact group $K$. Our construction is based on first defining the SRV for curves with values in Lie groups \cite{CeEsSch2015} and then lifting the curve in $M$ to a horizontal curve in the Lie group $G$. We compare our metric with the metric that has been considered in \cite{ZhSuKlLeSr2015,LeArBa2015,Lebrigant2016} and show the effectiveness of our algorithms in numerical examples using hurricane paths, i.e., curves with values on the homogenous space $S^2$. In future work, we plan to generalize results on the existence of minimizing geodesics and optimal reparametrizations that are known to hold for curves in Euclidean spaces, to curves in homogeneous spaces.


\section{The SRV for the space of curves with values in a homogenous space}

Let $M=G/K$ be a homogeneous space, where $G$ is a Lie group and $K$ is a compact Lie subgroup of $G$. We will denote the Lie algebras of $G$ and $K$ by $\mathfrak{g}$ and $\mathfrak{k}$ respectively. 
Assume that $G$ is equipped with a left invariant Riemannian metric that is also bi-invariant with respect to $K$. This metric induces a Riemannian metric on $M$ that is invariant under the left action by $G$, see e.g. \cite{Petersen1998}. Furthermore we will denote the set of all absolutely continuous curves with values in a manifold $N$ by $AC([0,1], N)$ -- here $N$ will be either $M$ or $G$ --  and  by $\Gamma$ the group of orientation preserving reparametrizations, $\Gamma=\operatorname{Diff}^+([0,1])$. 

\subsection{Curves with values in  a Lie group $G$}
Following the SRVF (introduced by Srivastava et al. in \cite{SrKlJoJe2011}), we define the map 
\begin{align}
Q: AC([0,1], G)&\to G\times L^2([0,1],\mathfrak{g})\notag\\
Q(\alpha)&=(\alpha(0), q),
\end{align}
where 
\begin{align}
q(t)=\left\{  \begin{array}{lcr}
L_{\alpha(t)^{-1}}\dfrac{\alpha'(t)}{\sqrt{\|\alpha'(t)\|}} &\alpha'(t)\neq 0\\
0       &\alpha'(t)=0
\end{array} \right.
\end{align}
In this definition, the notation $L_{\alpha(t)^{-1}}$ refers to left translation applied to elements of $G$, as well as to tangent vectors. The $q$-map here is the same as the $q$-map defined in \cite{CeEsSch2015} (using right translation instead of left). We have the following proposition:
\begin{proposition}
The map $$Q:AC([0,1], G)\to G\times L^2([0,1],\mathfrak{g}),$$ defined above, is a bijection.
\end{proposition}
\begin{proof}
Let $\alpha\in AC([0,1], G)$ denote the preimage under $Q$ of a given $(\alpha_0, q)\in G\times L^2([0,1],\mathfrak{g})$. By definition, $\alpha$ is the unique solution of the initial value problem
$\alpha(0)=\alpha_0$, and $\alpha'(t)=L_{\alpha(t)}(\|q(t)\|q(t)).$ In the case of $G=\mathfrak{g}=\mathbb R^n$, existence and uniqueness of such an $\alpha$ was proved by Robinson in \cite{Rob2012}.
To present a detailed proof of this result in the case of a Lie group $G$ is outside of the scope of this contribution, and we postpone it to a future extended journal version of this article.
\end{proof}

Since we have already given $G$ a Riemannian metric, and $L^2([0,1],\mathfrak{g})$ has its own $L^2$ metric, we obtain a product Riemannian metric on $G\times L^2([0,1],\mathfrak{g})$. Furthermore, as $Q$ is a bijection, there exists  a smooth structure on $AC([0,1],G)$ such that $Q$ is in addition a diffeomorphism. We can then use this diffeomorphism to induce a Riemannian metric (and thus distance function) on $AC([0,1],G)$. Note that it has been shown in \cite{Bru2016} that the mapping $Q$ is not a diffeomorphism and consequently does not induce a Riemannian metric on $AC([0,1],G)$, if the former is equipped with its natural smooth structure.

Given $\alpha_1, \alpha_2\in AC([0,1], G)$, let	$Q(\alpha_1)=(\alpha_1(0),q_1)$ and	$Q(\alpha_2)=(\alpha_2(0),q_2)$.
Then the distance function on $AC([0,1],G)$ takes the form:
\begin{align}
d(\alpha_1, \alpha_2)=\left(d^2(\alpha_1(0), \alpha_2(0))+\|q_1-q_2\|^2\right)^{1/2}
\end{align}
where the $d$ on the right hand side of this equation is the geodesic distance on $G$.
Consider the action of the reparametrization group $\Gamma$ on $AC([0,1], G)$ by right composition and the action of $G$ on $AC([0,1], G)$ by left multiplication. Given $\gamma\in\Gamma$ and $g\in G$, the corresponding actions of $\gamma$ and $g$ on the product space $G\times L^2([0,1],\mathfrak{g})$ are as follows:

\begin{equation}
	g\bullet(\alpha_0, q)\star\gamma=\left(g\alpha_0,\  q\circ\gamma\sqrt{\gamma'}\right),
\end{equation}
where $(\alpha_0, q)\in G\times L^2([0,1], \mathfrak{g})$. $G$ acts by isometries, since the metric on $G$ was chosen to be left-invariant. The proof that $\Gamma$ acts by isometries is the same as in the $\mathbb R^n$ case (see \cite{SrKlJoJe2011}) and we omit it. Hence we have the following proposition.
\begin{proposition}
	The Riemannian metric on $AC([0,1],G)$ and the corresponding distance function are preserved by the action of $G$ and by the action of the reparameterization group $\Gamma$. 
	\end{proposition}

We will now derive the formula for the induced Riemannian metric on $AC([0,1], G)$.
Denote by $\langle\cdot,\cdot\rangle^G$ the metric on $G$. Given $\alpha\in AC([0,1], G)$ and $u\in T_{\alpha}AC([0,1], G)$, one can compute the differential of $Q$:
\begin{align}
Q_{*\alpha}: T_{\alpha}AC([0,1], G)&\to T_{(\alpha(0),q)}(G\times L^2([0,1],\mathfrak{g}))\notag\\
Q_{*\alpha}u&=\left(u(0),q_{*\alpha}u\right),
\end{align}
where $q_{*\alpha}: T_{\alpha}AC([0,1], G)\to T_qL^2([0,1],\mathfrak{g})$ and
\begin{align}
q_{*\alpha}u=&\|\alpha'\|^{1/2}D_s(u)-\frac{1}{2} \| \alpha'\|^{-3/2} \langle D_s u, \delta^l(\alpha)\rangle^G \delta^l(\alpha).
\end{align} 
Here $\delta^l(\alpha)=\alpha^{-1}\alpha'$ and $D_s(v)=\frac1{\|\alpha'\|}\delta_{*\alpha}^l(v)$.  For a proof of this computation we refer to \cite{CeEsSch2015}. 
The metric on the space $AC([0,1], G)$ is then obtained as the pullback of the natural product metric of $G\times L^2([0,1],\mathfrak{g})$ under $Q$:
\begin{proposition}
	Let $u, v$ be smooth tangent vectors with foot point an immersion $\alpha$. The pullback metric $\mathcal{G}$ on $AC([0,1], G)$ at the smooth immersion $\alpha$ is given by
	\begin{align}
	\mathcal{G}_{\alpha}(u, v)&=\langle Q_{*\alpha}u, Q_{*\alpha}v\rangle_{Q(\alpha)}
	=\langle u(0),v(0)\rangle^G\\
	&+\int \langle D_s u^N, D_s v^N\rangle^G+\frac14\langle D_s u^T, D_sv^T\rangle^Gds,\notag
	\end{align}
	where we integrate with respect to arclength $ds=\|\alpha'(t)\|dt$,  $D_su^T=\langle D_su,
	\frac{\delta^l(\alpha)}{\|\alpha'\|}\rangle^G\left(\frac{\delta^l(\alpha)}{\|\alpha'\|}\right)$ and $D_su^N=D_su-D_su^T$
	are the tangential component and the normal component of $D_su$ respectively.
\end{proposition}
For $G=\mathbb R^d$ the formula for the metric $\mathcal{G}$ reduces to \eqref{elastic_metric}, i.e., we obtain the elastic metric as defined in \cite{MiSrJo2007}.
On Lie groups the last two terms form the pullback metric obtained by Celledoni et al. in \cite{CeEsSch2015} (using right  instead of left trivialization). However, it is different than the metric introduced 
 by Le Brigant et al. \cite{LeArBa2015} and Zhang et al. \cite{ZhSuKlLeSr2015} for arbitrary Riemannian manifolds. In our method the velocities are transported to the Lie algebra using left translation, while the metric in the above mentioned work is based on parallel transport. Thus these metrics will be different if $G$ is not an abelian Lie group. In Fig. \ref{H2examples} we show examples of geodesics for curves in hyperbolic space. The resulting geodesics are very similar to the geodesics obtained in \cite{LeArBa2015,Lebrigant2016}. We plan to further investigate the similarities between these methods in future work.

\begin{figure*}
	\begin{center}
		\includegraphics[width=0.225\linewidth]{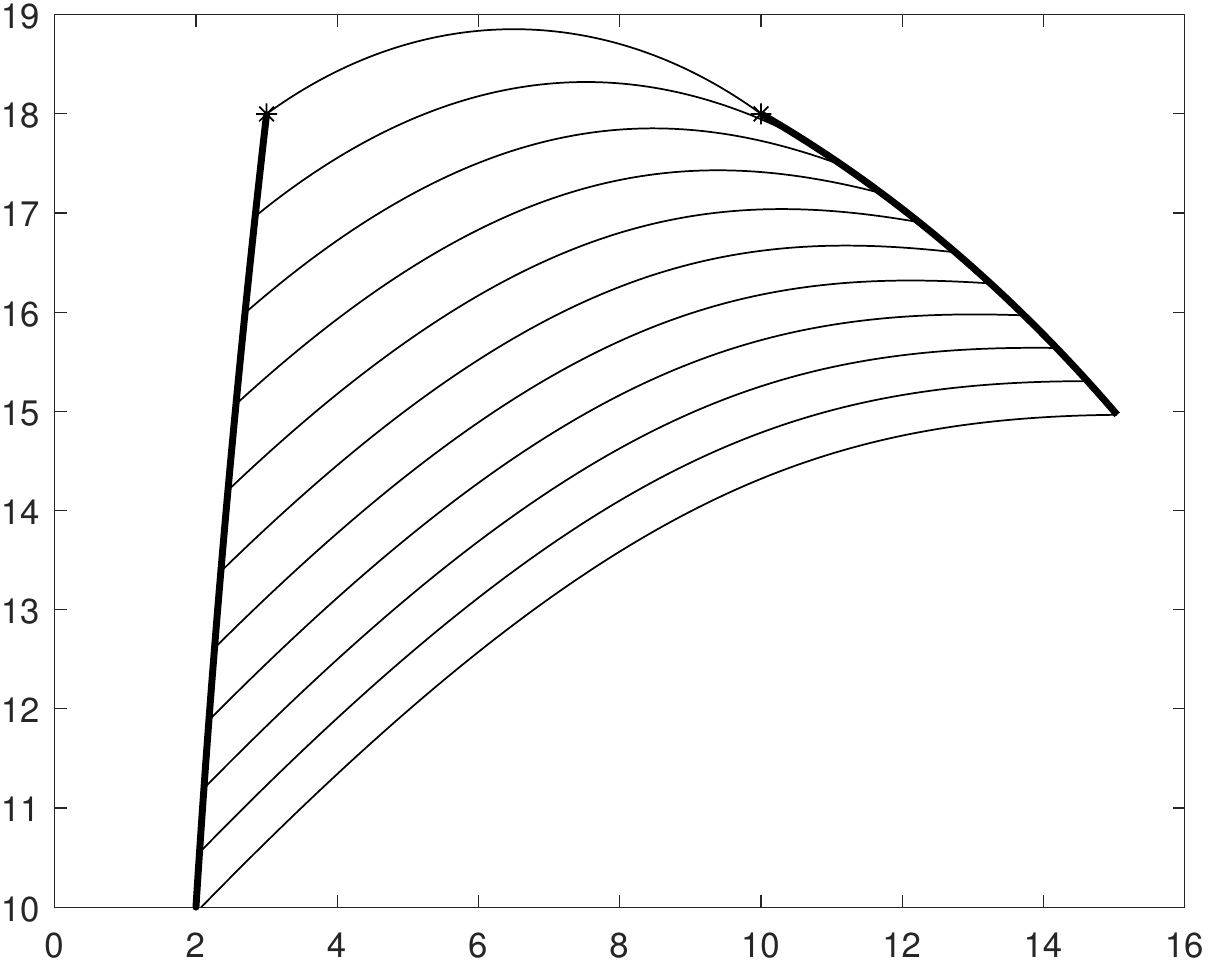}
		\includegraphics[width=0.225\linewidth]{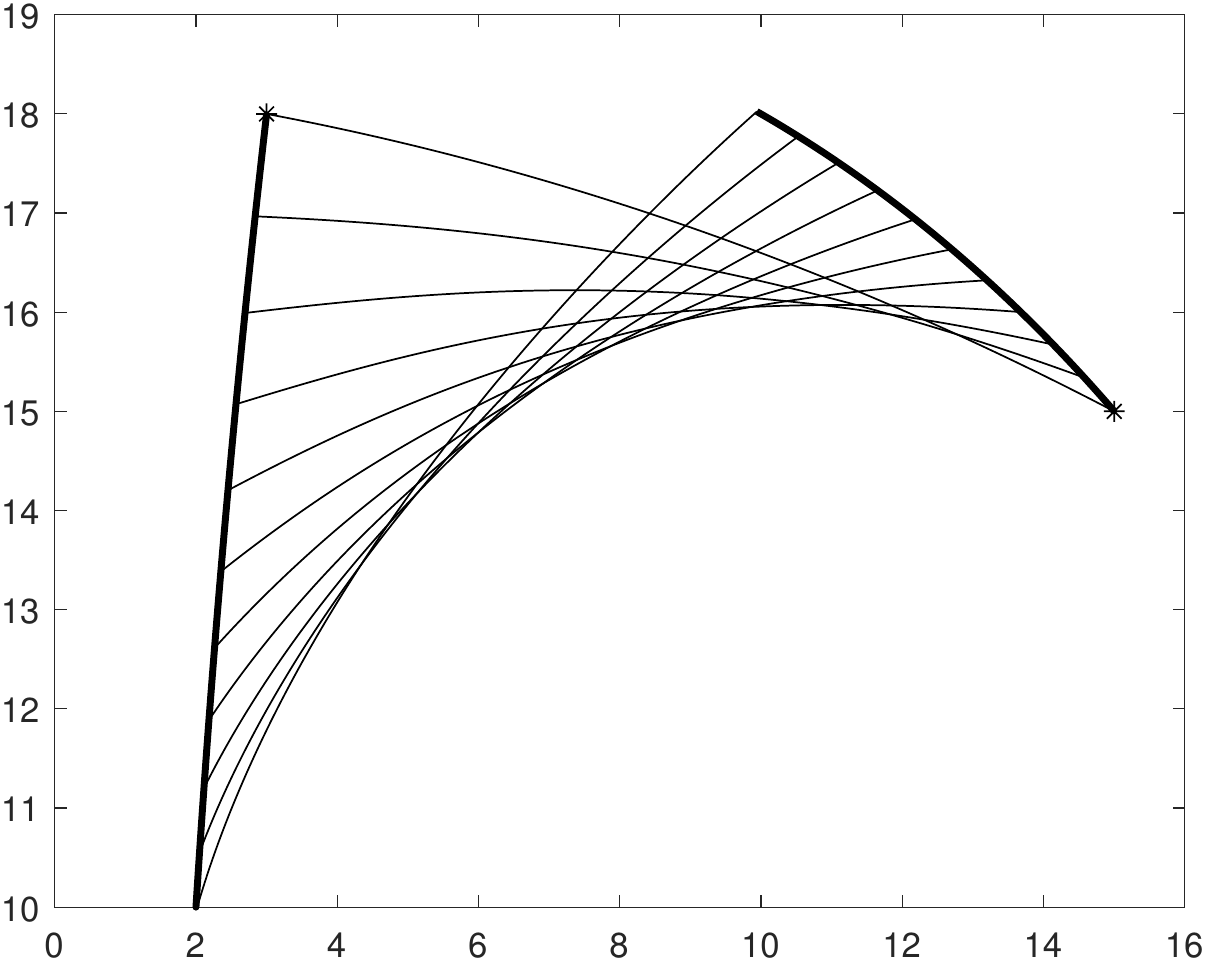}
		\includegraphics[width=0.225\linewidth]{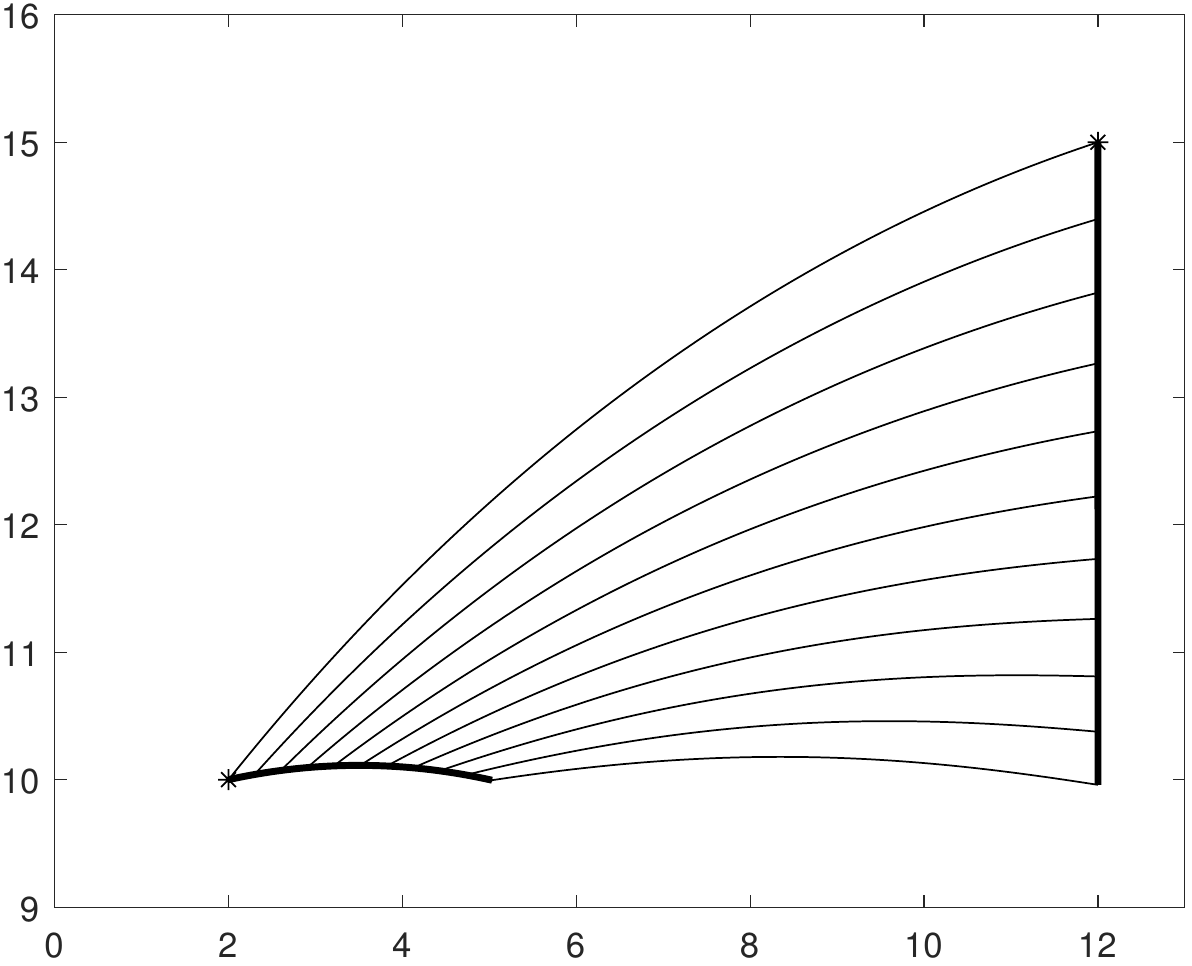}
		\includegraphics[width=0.225\linewidth]{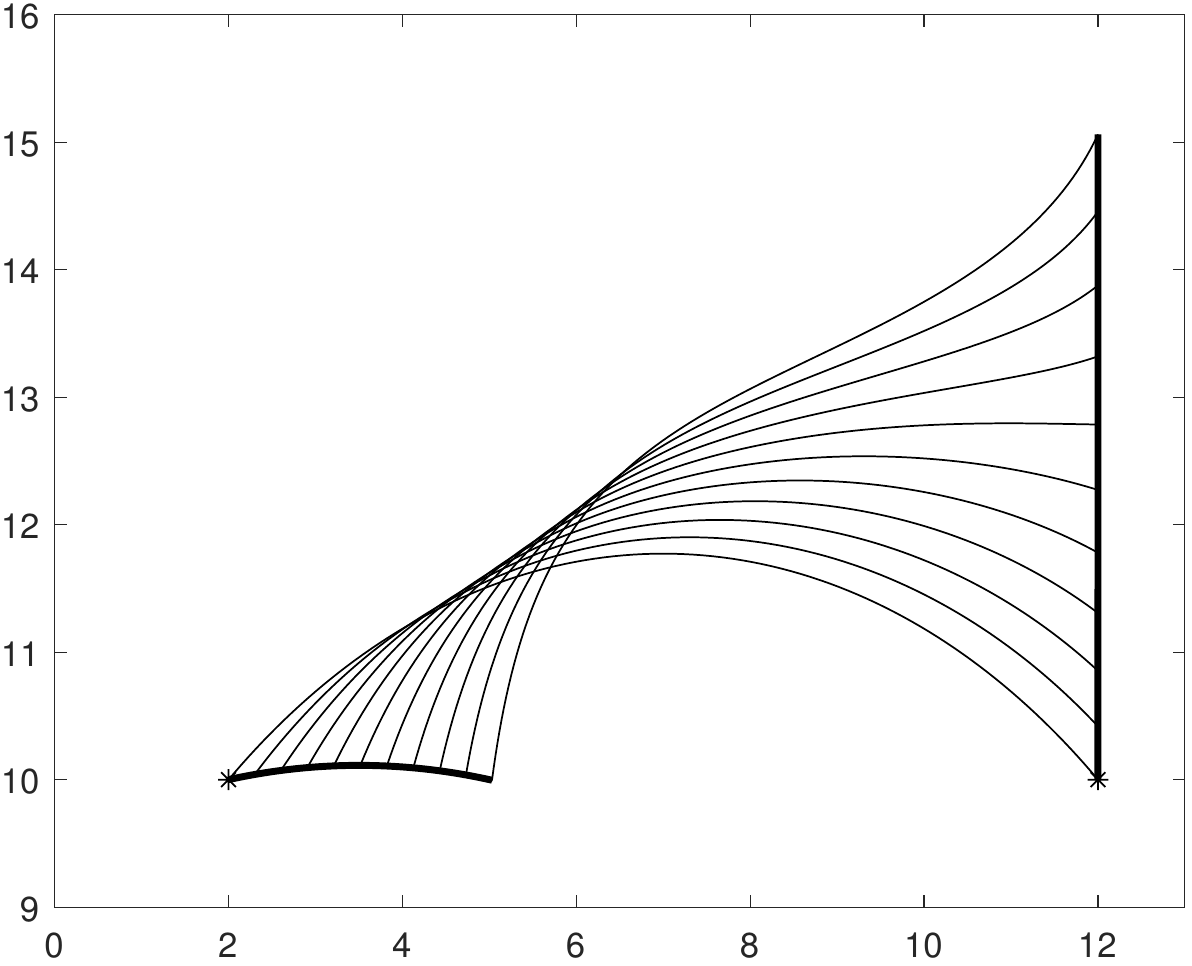}
	\end{center}
	\caption{Examples of geodesics between parametrized curves in 2-dimensional hyperbolic space. We show selected particle paths of the geodesic  connecting the boundary curves.}
	\label{H2examples}
\end{figure*}

\subsection{Curves with values in a homogeneous space $M$}

In this section, we will analyze curves in $M=G/K$ by relating them to their horizontal lifts in $G$. Note that $\mathfrak{g}=\mathfrak{k}\oplus \mathfrak{k}^{\perp}$, where  $\mathfrak{k}$ denotes the Lie algebra of $K$ and $\mathfrak{k}^{\perp}$ denotes the orthogonal complement of $\mathfrak{k}$ in $\mathfrak{g}$. Denote by $AC^{\perp}([0,1], G)$ the set of all absolutely continuous paths in $G$ which are orthogonal to each coset of $K$ that they meet. Since the metric on $G$ is left invariant, $\alpha\in AC^{\perp}([0,1], G)$ is equivalent to $L_{\alpha^{-1}}\alpha'(t)\perp\mathfrak{k}$, which is equivalent to $q\in L^2([0,1], \mathfrak{k}^{\perp})$, where $(\alpha(0), q) = Q(\alpha)$. Therefore, $Q$ restricts to a bijection between $AC^{\perp}([0,1], G)$ and $G\times L^2([0,1],\mathfrak{k}^{\perp})$.

Now consider the right action of $K$ on $G\times L^2([0,1], \mathfrak{k}^{\perp})$ given by:
\begin{equation}
	(\alpha_0, q)\ast y= (\alpha_0y, y^{-1}qy),
\end{equation}
where $y\in K, \alpha_0\in G$ and $q\in L^2([0,1],\mathfrak{k}^{\perp})$. 
Note that $K$ acts by isometries, where we put the standard $L^2$ metric on $L^2([0,1], \mathfrak{k}^{\perp})$ and the product metric on $G\times L^2([0,1], \mathfrak{k}^{\perp})$.
Denote by $\pi: G\to M$ the quotient map, $V_p=\ker \pi_{*p}$ the vertical distribution for $p\in G$ and $H_p$ the orthogonal complement of $V_p$ in $T_pG$. For every $p\in G$, $T_pG=H_p\oplus V_p$ and $\pi_{*p}$ induces an isomorphism between $H_p$ and $T_{\pi(p)}M$. Thus, given $\beta\in AC([0,1], M)$ and $\alpha_0\in \pi^{-1}(\beta(0))$, there is a unique lift $\alpha\in  AC^{\perp}([0,1], G)$ such that $\alpha(0)=\alpha_0$ and $\beta(t)=\pi(\alpha(t))$. Note that the horizontal lift of $\beta$ to $\alpha\in AC^{\perp}([0,1], G)$ depends only on the choice of the lift $\alpha_0$ of the initial point $\beta(0)$:
\begin{align}
\xymatrix{	& G \ar[d]^{\pi} \\
	I \ar[ur]^{\alpha} \ar[r]_{\beta} & M= G/K}
\end{align}

Let $\alpha_0, \tilde{\alpha}_0$ be two lifts of $\beta(0)$ and $\alpha, \tilde{\alpha}$ be the lifts of $\beta$ in $AC^{\perp}([0,1], G)$ starting at $\alpha_0$ and $\tilde{\alpha}_0$ respectively. Then $\tilde{\alpha}=\alpha y,$
where $y=\alpha_0^{-1}\tilde{\alpha}_0\in K$; also
\begin{equation}
 (\tilde\alpha_0, \tilde q)=(\alpha_0, q)\ast y,
\end{equation}
where $(\alpha_0, q)= Q(\alpha)$ and $(\tilde\alpha_0, \tilde q)=Q(\tilde\alpha)$.
It follows that $Q$ induces a bijection
\begin{equation}
	(G\times L^2([0,1], \mathfrak{k}^{\perp}))/K \to AC([0,1], M).
\end{equation}

Since $K$ is compact and acts freely on $G\times L^2([0,1],\mathfrak{k}^{\perp})$, there is an inherited Riemannian metric on the space $(G\times L^2([0,1], \mathfrak{k}^{\perp}))/K$. A minimal geodesic in the quotient corresponds to a shortest geodesic between two orbits in $G\times L^2(I,\mathfrak{k}^{\perp})$ under the action of $K$. We use this bijection to transfer the smooth structure and the Riemannian metric on $(G\times L^2(I,\mathfrak{k}^{\perp}))/K$ to $AC([0,1], M)$, making the latter into a Riemannian manifold.

Suppose $\beta_1, \beta_2$ are two paths in $AC([0,1], M)$; let $\alpha_1$ and $\alpha_2$ be lifts of $\beta_1$, $\beta_2$ in $AC^{\perp}([0,1], G)$. Let
\begin{equation}
	Q(\alpha_1)=(\alpha_1(0),q_1),\quad Q(\alpha_2)=(\alpha_2(0),q_2).
\end{equation}
The distance between $\beta_1$ and $\beta_2$ induced from the distance function on $AC([0,1], G)$ is given by:
\begin{align}
  &d(\beta_1, \beta_2)\notag\\
    =&\inf_{y\in K}\left(d^2(\alpha_1(0), \alpha_2(0)y)+\|q_1-y^{-1}q_2y\|^2\right)^{1/2}.
\end{align}
Consider now the right action of $\Gamma$ and the left action of $G$ on $G\times L^2([0,1],\mathfrak{k}^{\perp})$. Similar as in the case of $AC([0,1], G)$, we have the following proposition:
\begin{proposition}\label{Minvariance}
The Riemannian metric on $AC([0,1],M)$ and the corresponding distance function are preserved by the action of $G$ and by the action of the reparameterization group $\Gamma$.
\end{proposition}
The formula for the induced pull back metric on the space $AC([0,1], M)$ is then simply given by restricting the metric $\mathcal{G}$ to horizontal vector fields. 


\section{Computing Geodesics}

\subsection{Comparing Curves in $M$}

To compute the geodesic between $\beta_1$ and $\beta_2$ in $AC([0,1],M)$, we need to compute the geodesic of minimal length between the orbits of $Q(\alpha_1)$ and $Q(\alpha_2)$ under the action of $K$. To do this, we need to find $y\in K$ that minimizes
\begin{align}
d^2(\alpha_1(0), \alpha_2(0)y)+\|q_1-y^{-1}q_2y\|^2.
\end{align} 
Then the geodesic between
$(\alpha_1(0),q_1)$ and $(\alpha_2(0)y,y^{-1}q_2y)$ will project to a geodesic between $\beta_1$ and $\beta_2$, see \cite{Oneill1966,Mic2016} for more details regarding Riemannian submersions. 

In practice we will search for the optimal $y$ using a gradient descent method. 
Towards this aim we define the functional $F: K\to \mathbb{R}$ by
\begin{equation}
	F(y)=d^2(\alpha_1(0), \alpha_2(0)y)+\|q_1-y^{-1}q_2y\|^2,
\end{equation}
which is the square of the distance function between the $Q$-map $(\alpha_1(0), q_1)$ and $(\alpha_2(0)y, y^{-1}q_2y)$. Since $K$ acts transitively on $(\alpha_2(0), q_2)\ast K$, we can just calculate the gradient at $y=I$.

To simplify the presentation, we   assume that $G$ is a matrix group and that the inner product on the 
Lie algebra $\mathfrak{g}$ is given by $\langle x, y\rangle = \operatorname{tr}(x^ty)$, where $x^t$ means the transpose of $x$.
We will calculate the gradient of the two terms of $F$ separately.  
The first term of $F$ can be  extended to a function $F_1: G\to \mathbb{R}$, defined by the same formula: $F_1(y)= d^2\left(\alpha_1(0), \alpha_2(0)y\right)$. By left invariance of the metric on $G$, we can rewrite this as $F_1(y)=d^2\left(\alpha_2(0)^{-1}\alpha_1(0), y\right)$. It is a well-known fact that the gradient of this function at $y=I$ is given by $\nabla_I F_1=-2\textrm{Log}_I(\alpha_2(0)^{-1}\alpha_1(0))\in\mathfrak{g}$, where Log denotes the inverse Riemannian exponential function at $I\in G$. If Log is multivalued, we will take the value with the smallest norm. Now, if we restrict $F_1$ to $K$, then the gradient in $\mathfrak{k}$ will simply be the projection of the above expression from $\mathfrak{g}$ to $\mathfrak{k}$. Thus the gradient of the first term of $F(y)$ is given by $-2\textrm{Proj}_{\mathfrak{k}}\left(\textrm{Log}_I(\alpha_2(0)^{-1}\alpha_1(0))\right).$

Now we turn our attention to the second term of $F$. By the bi-invariance of the metric under multiplication by elements of $K$, we have
\begin{align}
F_2(y)&=\|q_1-y^{-1}q_2y\|^2\notag\\
&=\|q_1\|^2+\|y^{-1}q_2y\|^2-2\langle q_1, y^{-1}q_2y\rangle\notag\\
&=\|q_1\|^2+\|q_2\|^2-2\langle q_1, y^{-1}q_2y\rangle.
\end{align}
Since the first two terms do not depend on $y$, we just need to calculate the gradient of $\langle q_1, y^{-1}q_2y\rangle$. We will use the first order approximation $y\sim I+tV$ and $y^{-1}\sim I-tV$, where $V\in \mathfrak{k}$. Then we have the directional derivative of $\langle q_1, y^{-1}q_2y\rangle$ at $I$ in the direction $V$:
\begin{align}
\dfrac{d}{dt}_{t=0}\langle q_1, &(I-tV)q_2(I+tV)\rangle=\langle q_1, q_2V\rangle-\langle q_1, Vq_2\rangle\notag\\
&=\int_0^1\operatorname{tr}(q_1^tq_2V)dt-\int_0^1\operatorname{tr}(q_1^tVq2)dt\notag\\
&=\int_0^1\operatorname{tr}(q_1^tq_2V)dt-\int_0^1\operatorname{tr}(q_2q_1^tV)dt\notag\\
&=\operatorname{tr}(\int_0^1(q_1^tq_2-q_2q_1^t)dt V)\notag\\
&=\langle\int_0^1(q_2^tq_1-q_1q_2^t)dt,V\rangle
\end{align}
So we get the gradient of the second term at $y=I$
\begin{equation}
	2\textrm{Proj}_{\mathfrak{k}}\left(\int_0^1(q_1q_2^t-q_2^tq_1)dt\right).
\end{equation}
Hence the gradient of $F$ at $y=I$ is
\begin{align}
\nabla_I F=&2\textrm{Proj}_{\mathfrak{k}}(-\textrm{Log}_I(\alpha_2(0)^{-1}\alpha_1(0))\notag\\
&\qquad\qquad\qquad+\int_0^1(q_1q_2^t-q_2^tq_1)dt).
\end{align}

This yields the following algorithm to obtain the geodesic of minimal length between the orbits of the $Q$-map $(\alpha_1(0), q_1)$ and $(\alpha_2(0), q_2)$:
\begin{enumerate}[(1)]
	\item For $(\alpha_1(0), q_1), (\alpha_2(0), q_2)\in G\times L^2([0,1], \mathfrak{k}^{\perp})$, set the step size $\epsilon$ and calculate the gradient at $y=I$.
	\item Update $(\alpha_2(0), q_2)$ to $(\alpha_2(0)y, y^{-1}q_2y)$, where $$y=\textrm{Exp}_I(-\epsilon\nabla F).$$
	\item If the norm of $\nabla F$ is small enough, then stop. Otherwise go back to step (1).
\end{enumerate}

Using this algorithm, we obtain a geodesic which locally minimizes the distance between the orbits of $(\alpha_1(0), q_1)$ and $(\alpha_2(0), q_2)$. If we start with a different orbit representative $\tilde{\alpha}_2(0)\in G$, the algorithm may converge to a different geodesic. To increase our chance of finding the minimal geodesic, we act on $(\alpha_2(0), q_2)$ by several different elements in $K$, and use the resulting points as initializations for the gradient algorithm. 


\begin{figure*}
	\begin{center}
		\includegraphics[width=0.225\linewidth]{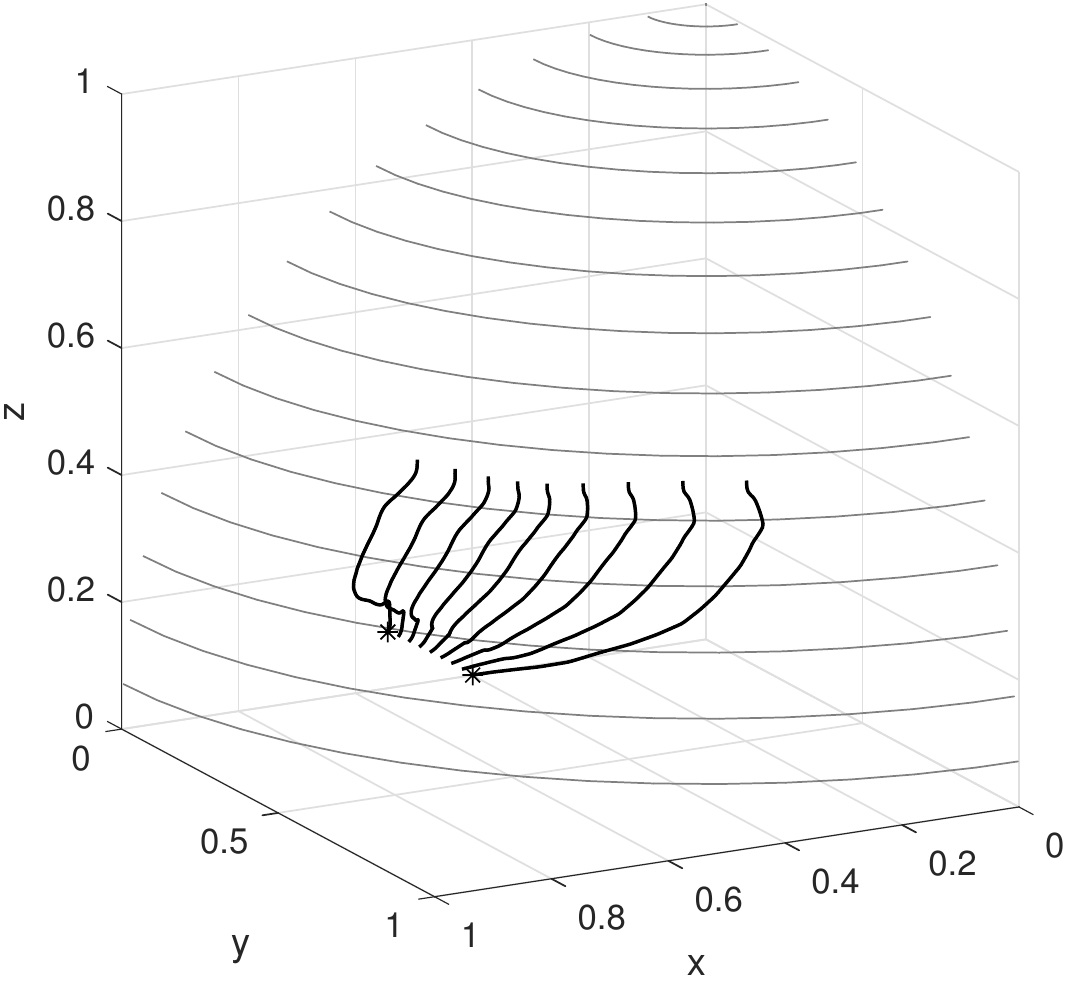}
		\includegraphics[width=0.225\linewidth]{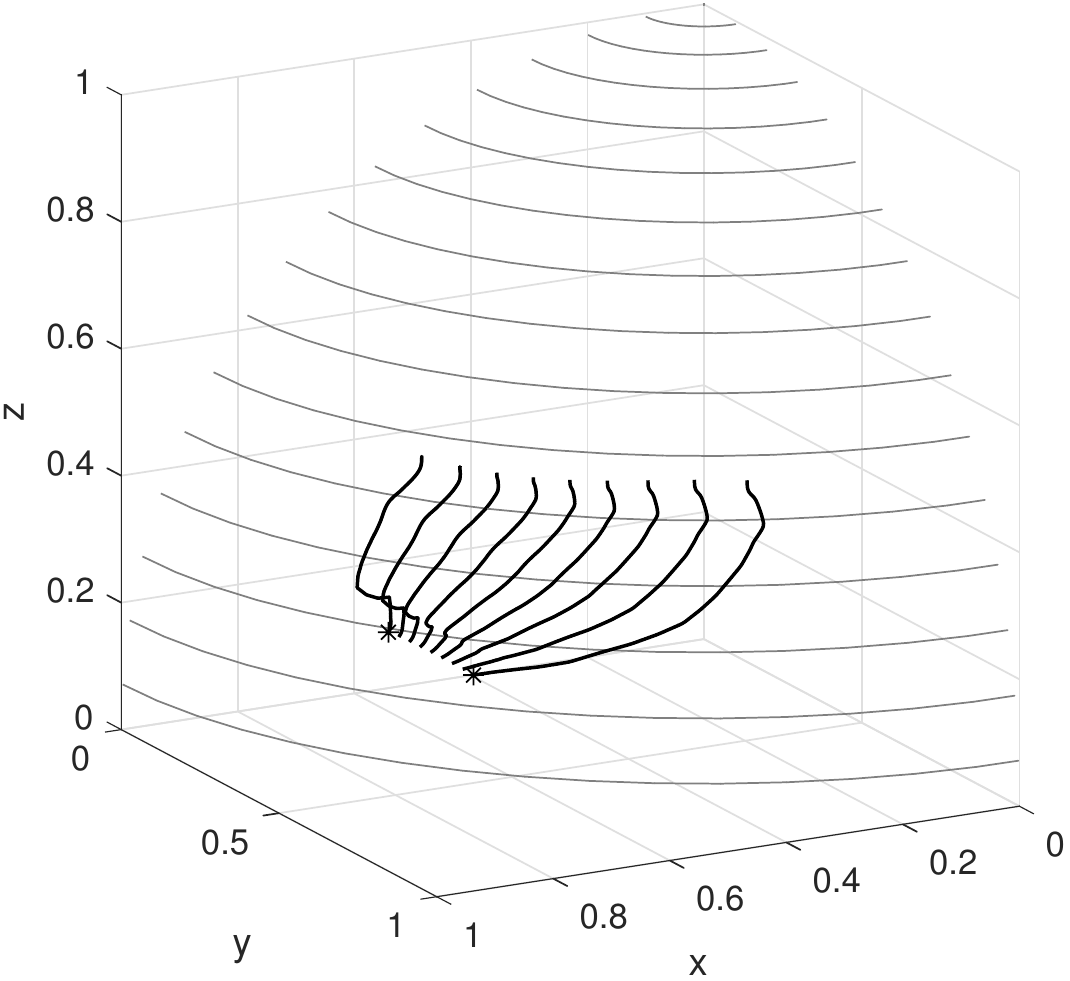}
		\includegraphics[width=0.225\linewidth]{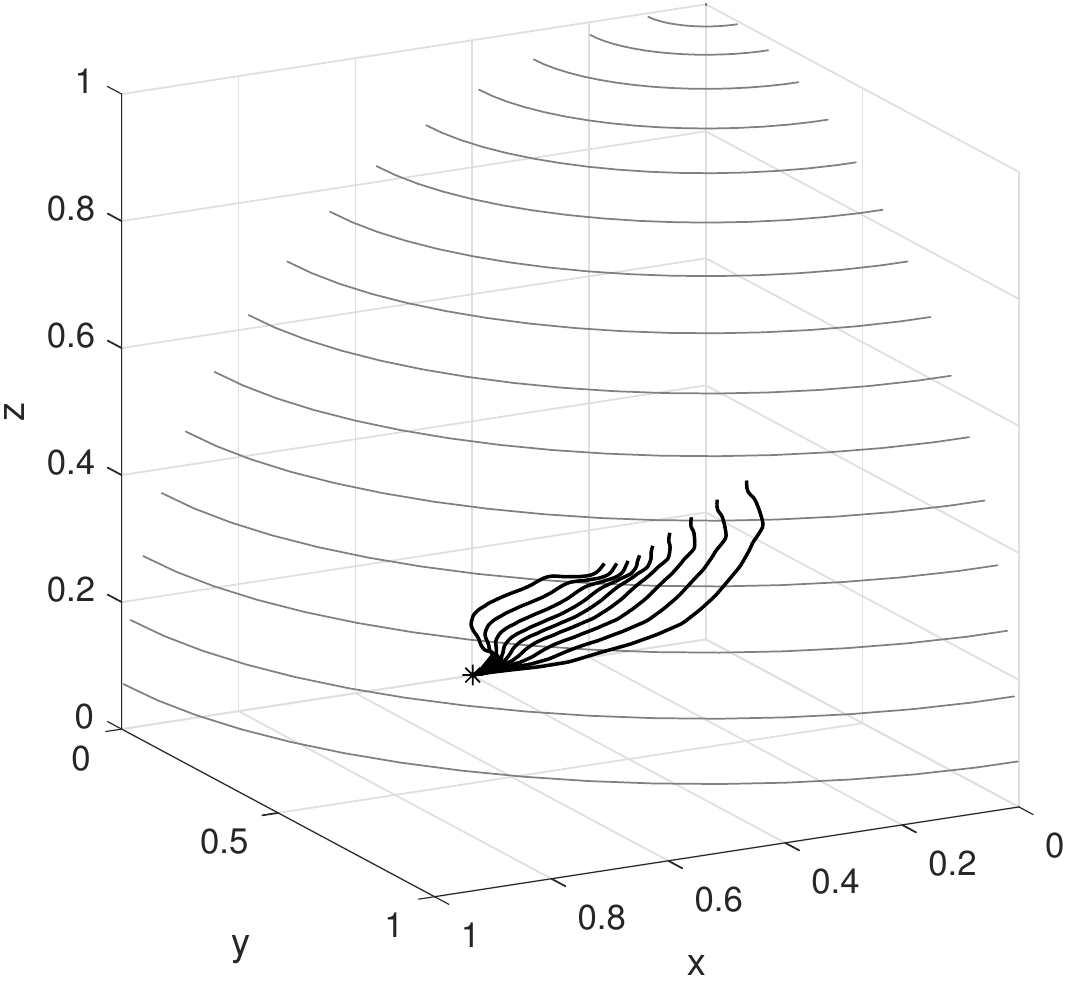}
		\includegraphics[width=0.225\linewidth]{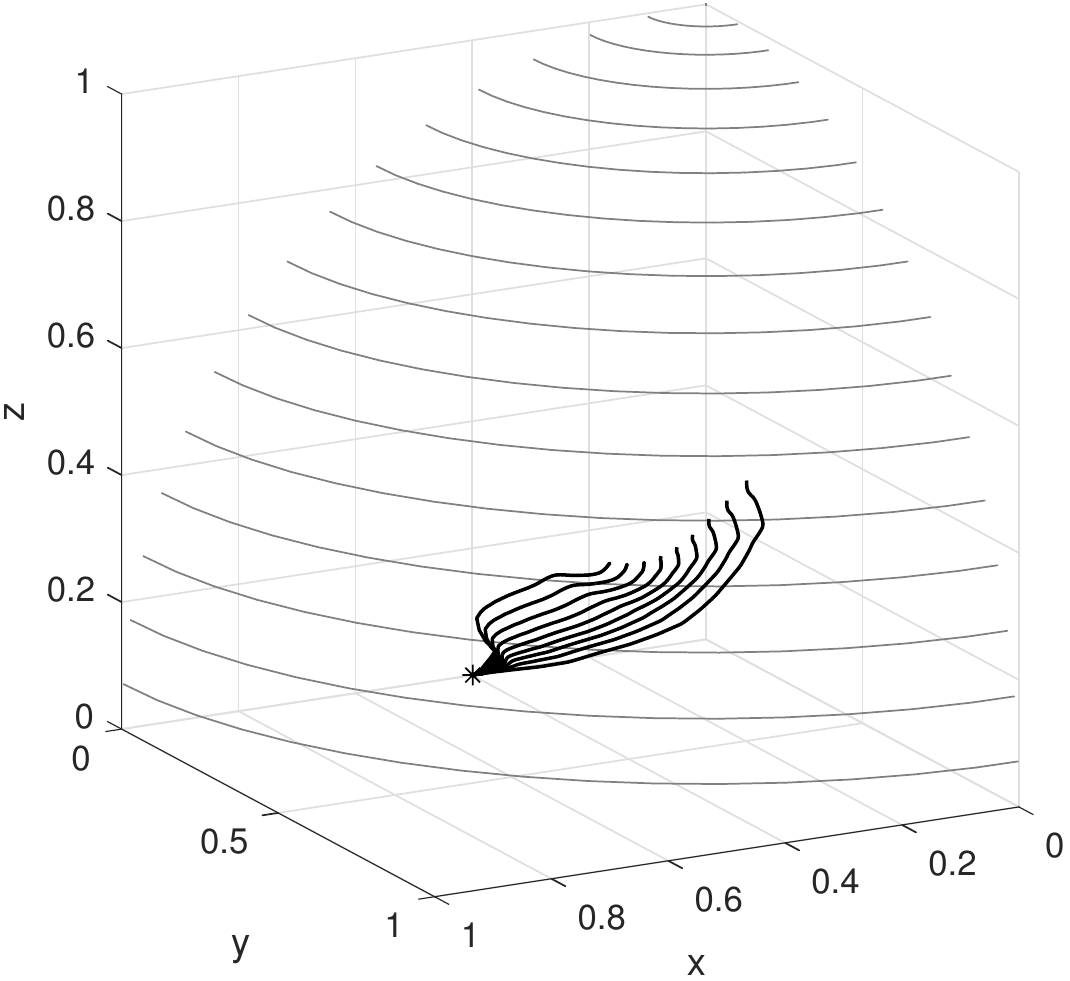}\\
		
		\includegraphics[width=0.225\linewidth]{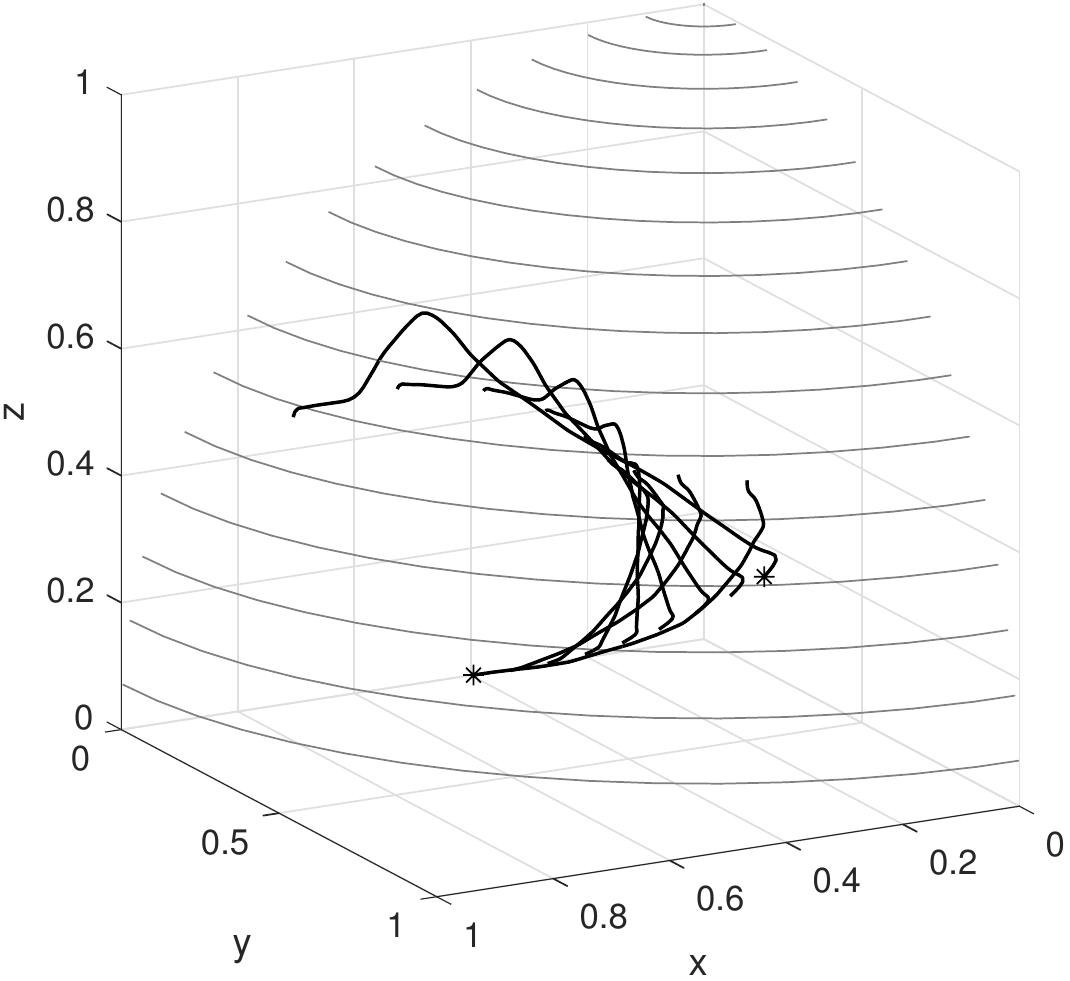}
		\includegraphics[width=0.225\linewidth]{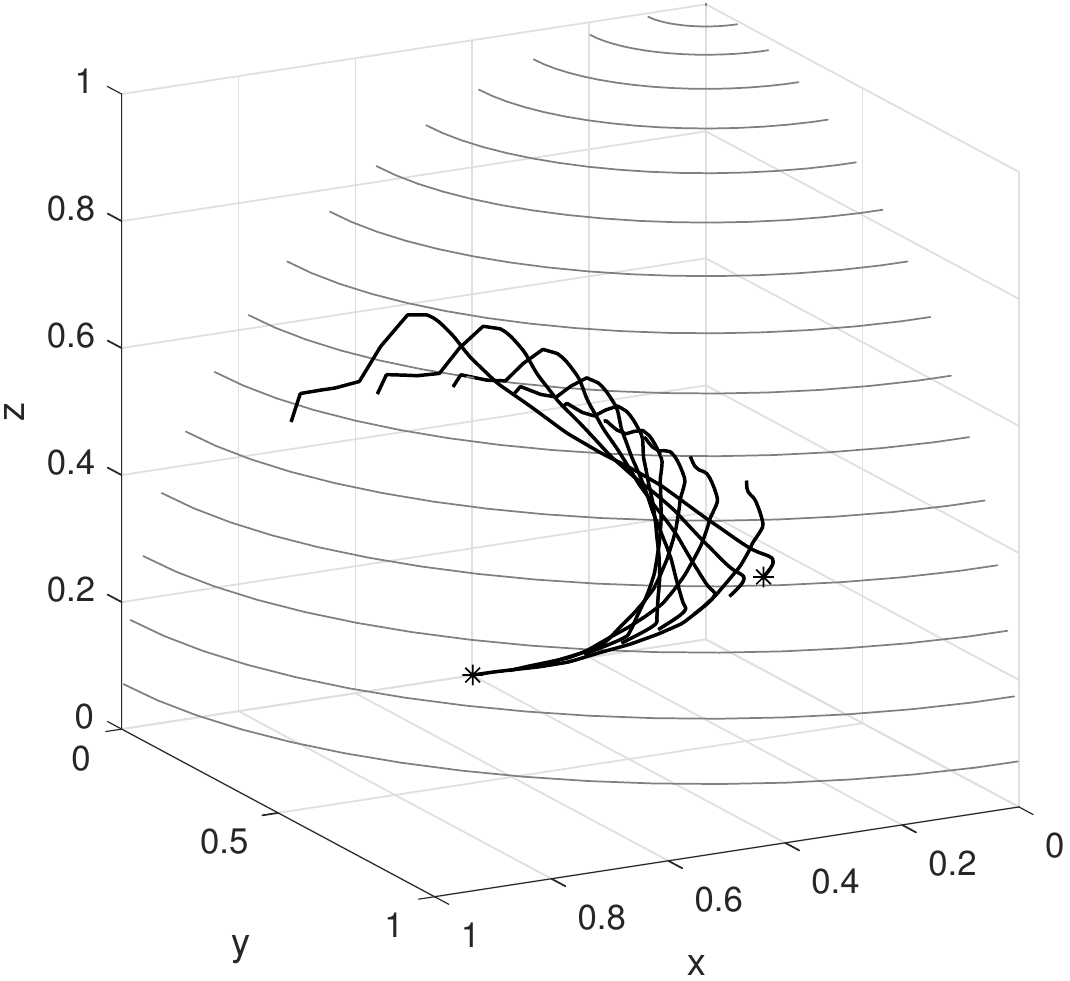}
		\includegraphics[width=0.225\linewidth]{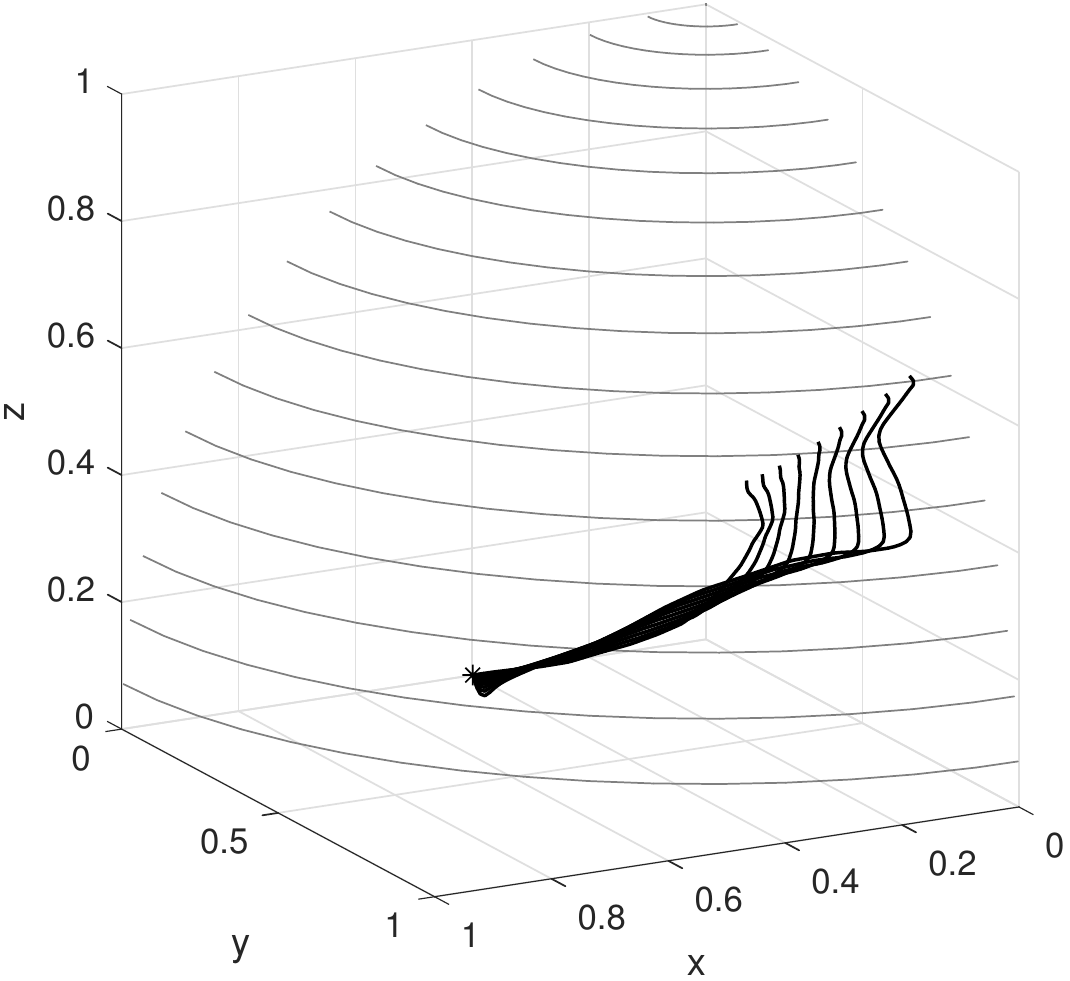}
		\includegraphics[width=0.225\linewidth]{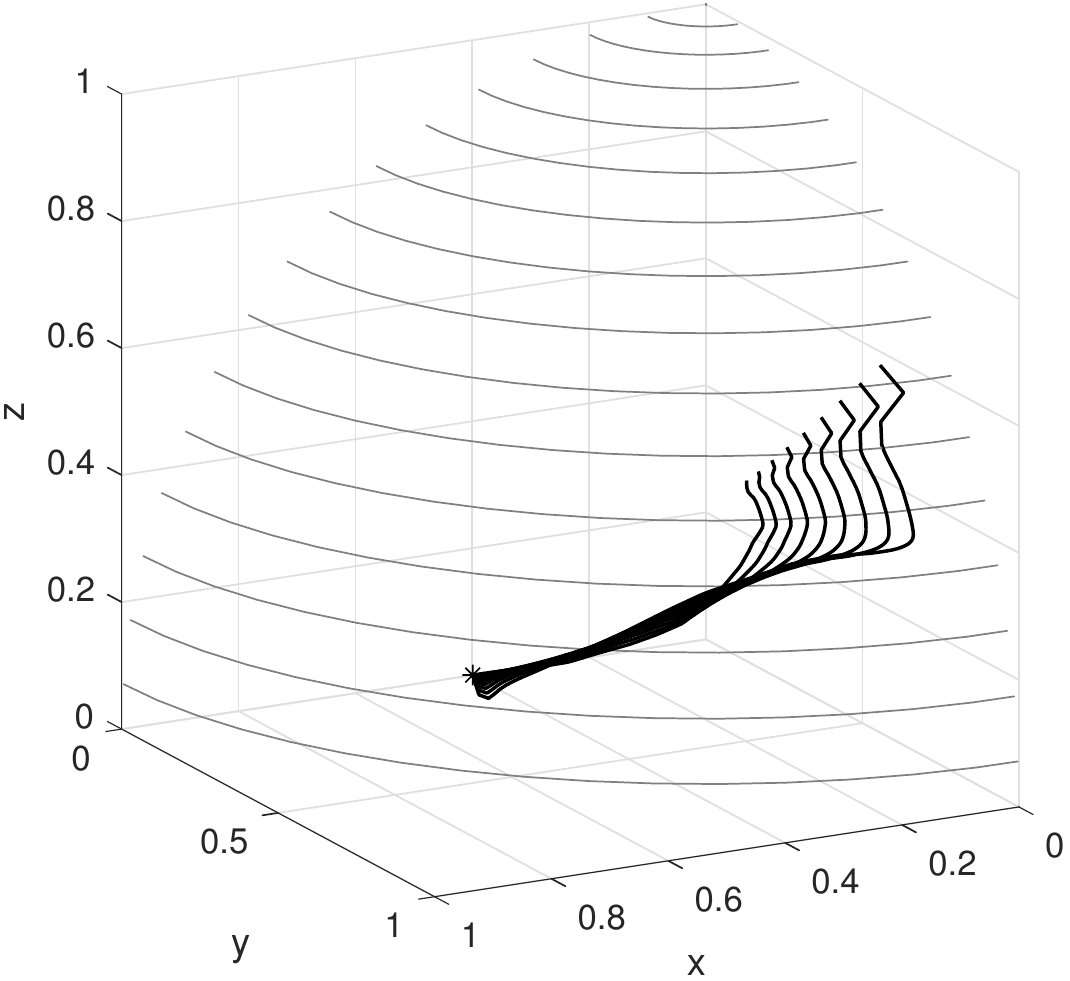}
	\end{center}
	\caption{Examples of geodesics between two curves in $AC([0,1], S^2)$, $\mathcal{S}^{AC}([0,1], S^2)$, $AC([0,1], S^2)/\operatorname{SO}(3)$, $\mathcal{S}^{AC}([0,1], S^2)/\operatorname{SO}(3)$. Starting points of the curves are marked with a~$\star$.}
	\label{fig:short}
\end{figure*}


\subsection{Comparing Curves in $M$ up to Reparameterization}\label{ModRepa}
In the following, we are interested in comparing the shape of unparameterized curves. To mathematically 
formulate our matching problem we define an equivalence relation on the space $AC([0,1], M)$ as follows: Given elements $\beta_0$ and $\beta_1$ of $AC([0,1], M)$, define $\beta_0\sim\beta_1$ if and only if their orbits $\beta_0\Gamma$ and $\beta_1\Gamma$ have the same closure  with respect to the geodesic distance metric on $AC([0,1], M)$ we defined before, see \cite{LaRoKl2015}. We then define the shape space 
$\mathcal{S}^{AC}([0,1],M)$ to be the set of equivalence classes under $\sim$ and we refer to the  ``shape" of a curve as its equivalence class in  
$\mathcal{S}^{AC}([0,1], M)$.  
The space $\mathcal{S}^{AC}([0,1], M)$ is not a manifold, but we can endow $\mathcal{S}^{AC}([0,1], M)$ with a metric so that it becomes a metric space \cite{LaRoKl2015,Bru2016}.

By proposition \ref{Minvariance}, the distance function $d$ on $AC([0,1],M)$ is reparametrization invariant. We now consider the quotient space  $\mathcal{S}^{AC}([0,1], M)= AC([0,1], M)/\Gamma$. The induced distance is defined to be
\begin{align}
&d([\beta_1], [\beta_2])\notag\\
=&\inf_{\substack{y\in K\\ \gamma\in\Gamma}}\left(d^2(\alpha_1(0), \alpha_2(0)y)+\|q_1-y^{-1}(q_2,\gamma)y\|^2\right)^{1/2}.
\end{align}
To compute the geodesic between $[\beta_1]$ and $[\beta_2]$, we need to find the optimal $y\in K$ and $\gamma\in\Gamma$ to minimize this distance.  
Since the action of $K$ and the action of $\Gamma$ on $G\times L^2(I,\mathfrak{k}^{\perp})$ commute with each other, we use the gradient method and the dynamic programming algorithm \cite{SrKl2016} to obtain a satisfactory convergence of both $y$ and $\gamma$. 

For curves in values in $\mathbb R^d$,
 the existence of optimal repara\-metrizations has been shown for PL curves \cite{LaRoKl2015}  and $C^1$ curves \cite{Bru2016}. In future work we plan to generalize these results to curves with values in homogenous spaces.

\subsection{Comparing Curves in $M$ up to Rigid Motions}\label{ModRigidMotion}
The group $G$ acts on $M$ as its group of rigid motions. For certain applications, we might want to calculate distances and geodesics in the quotient space $AC([0,1],M)/G$. By proposition \ref{Minvariance}, the distance function on $AC([0,1],M)$ is invariant under the left action of $G$.
On the space $AC([0,1],M)/G$, the distance function is then defined by 
\begin{align}
&d([\beta_1], [\beta_2])\notag\\
=&\inf_{\substack{y\in K\\ g\in G}}\left(d^2(\alpha_1(0), g\alpha_2(0)y)+\|q_1-y^{-1}q_2 y\|^2\right)^{1/2}.
\end{align}
Since $g$ appears only in the first summand in this formula, the minimum value of this distance can be achieved as follows: First choose $y\in K$ to minimize the second term, and then choose $g=\alpha_2(0)y\alpha_1^{-1}(0)$, which will result in $d(\alpha_1(0), g\alpha_2(0)y)=0$ and the simplified distance formula:
\begin{align}
d([\beta_1], [\beta_2])=\inf_{y\in K}\|q_1-y^{-1}q_2 y\|.
\end{align}
As a result, the gradient calculation is much simpler in this case.

\subsection{Comparing Curves in $AC([0,1],M)$ up to Rigid Motions and Reparametrizations}\label{ModRigidMotionRepa}
Now we consider both the action of $\Gamma$ and  of $G$ on $AC([0,1],M)$ simultanously. We want to mod out both of these actions and focus on the quotient $AC([0,1],M)/(G\times\Gamma)$. The distance function is defined to be
\begin{align}
&d([\beta_1], [\beta_2])\notag\\
=&\inf_{\substack{y\in K\\ \gamma\in\Gamma, g\in G}}\left(d^2(\alpha_1(0), g\alpha_2(0)y)+\|q_1-y^{-1}(q_2,\gamma) y\|^2\right)^{1/2}.
\end{align}
Similar as in section \ref{ModRigidMotion}, the distance function can be simplified to
\begin{align}
d([\beta_1], [\beta_2])
=\inf_{\substack{y\in K\\ \gamma\in\Gamma}}\|q_1-y^{-1}(q_2,\gamma) y\|.
\end{align}
The optimal $y\in K$ and $\gamma\in\Gamma$ can now be found using the gradient and the dynamic programming algorithm, and we then set $g=\alpha_2(0)y\alpha_1^{-1}(0)$.

\section{Curves with values on $S^n$}
In this section we want to describe the important special case of curves with values on $S^n$ in more detail. 
To view the sphere as a homogenous space we consider the  Lie group
\begin{equation}
	\operatorname{SO}(n)=\{A\in GL(n,\mathbb{R})| A^tA=AA^t=I, \det(A)=1\}
\end{equation} with corresponding Lie algebra
\begin{equation}
	\mathfrak{so}(n)=\{X\in M(n,\mathbb{R})| X+X^t=0\}.
\end{equation}
It is well known that $S^n\cong\operatorname{SO}(n+1)/\operatorname{SO}(n)$, 
where we identify $\operatorname{SO}(n)$ as a subgroup of $\operatorname{SO}(n+1)$ using the inclusion
$	A \to \left ( \begin{array}{lcr}
	A & 0 \\ 0 & 1
	\end{array}  \right )$.
Let ${\bf n}=(0,...0,1)\in S^n$ be the north pole. Then the quotient map $\pi: \operatorname{SO}(n+1)\to S^n$ is defined by $\pi(\alpha)=\alpha{\bf n}$.

To use the previously developed theory we need to define  a Riemannian metric on $\operatorname{SO}(n+1) $ that is bi-invariant with respect to  $\operatorname{SO}(n)$:
for any pair of tangent vectors $u$ and $v$ in $T_g\operatorname{SO}(n+1)$ with $g\in \operatorname{SO}(n+1)$, we define the inner product
$
	\langle u,v\rangle_g=\operatorname{tr}(u^tv).
$	
It is straightforward to check that this metric is indeed bi-invariant under multiplication by elements of $\operatorname{SO}(n+1)$ and thus in particular by multiplication with elements of $\operatorname{SO}(n)\subset\operatorname{SO}(n+1)$. Using the bi-invariance of the metric, the Riemannian exponential map at the identity is equal to the Lie group exponential \cite{Petersen1998} and is thus of the form $v\to e^v$. The inverse Riemannian exponential map at identity is  simply the log function $g\to \log(g)$. 

The following well-known lemma will be useful in calculating the lift of paths in $S^n$ to paths in $SO(n+1)$:
\begin{lemma}
	Let $p, q\in S^n$ and $p\neq-q$, then the most efficient rotation that takes $p\to q$ can be expressed as
	\begin{align}
		R_{p, q}=\left(I-\dfrac{2}{|p+q|^2}(p+q)(p^t+q^t)\right)(I-2pp^t).
	\end{align}
	By most efficient, we mean the rotation closest to $I$ with respect to the bi-invariant metric on $\operatorname{SO}(n+1)$.
\end{lemma}
This formula is only valid if $p\neq-q$, since if $p=-q$ there is no unique shortest rotation taking $p$ to $q$.
Using the above lemma we obtain the following algorithm for lifting paths: 
\begin{figure*}
	\begin{center}
		\includegraphics[width=0.233\linewidth]{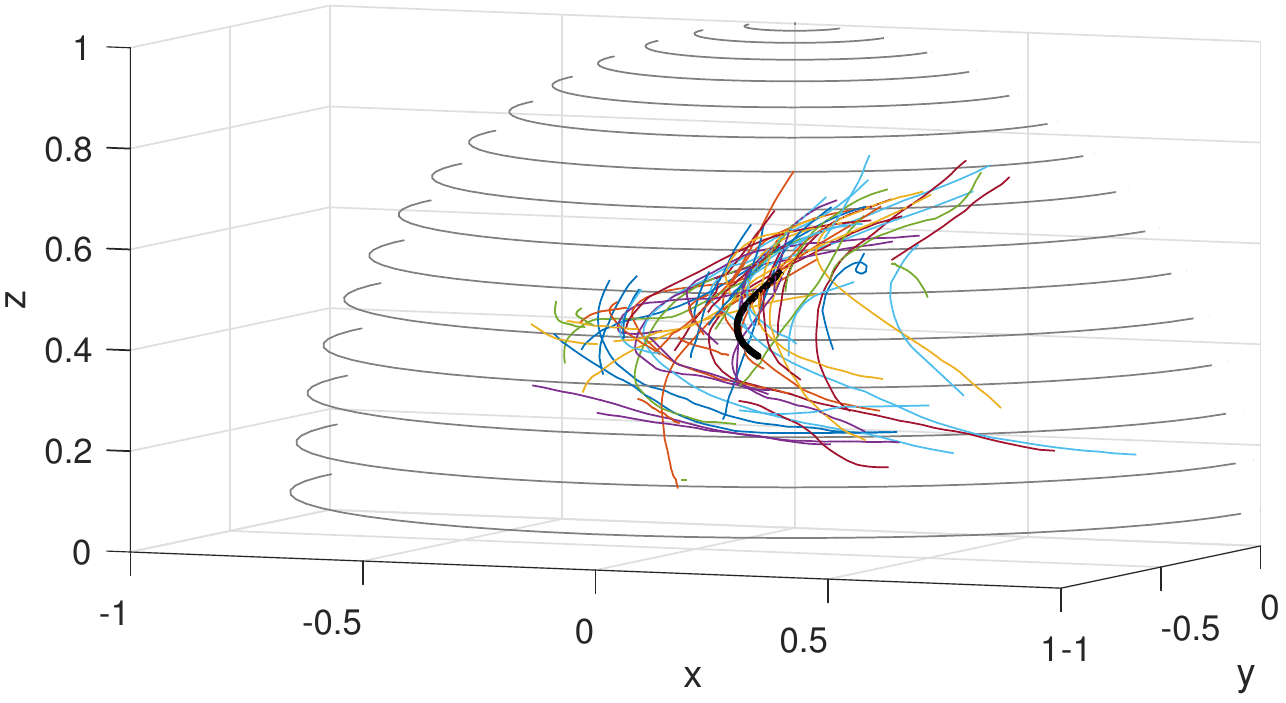}\hskip 0.5in
		\includegraphics[width=0.233\linewidth]{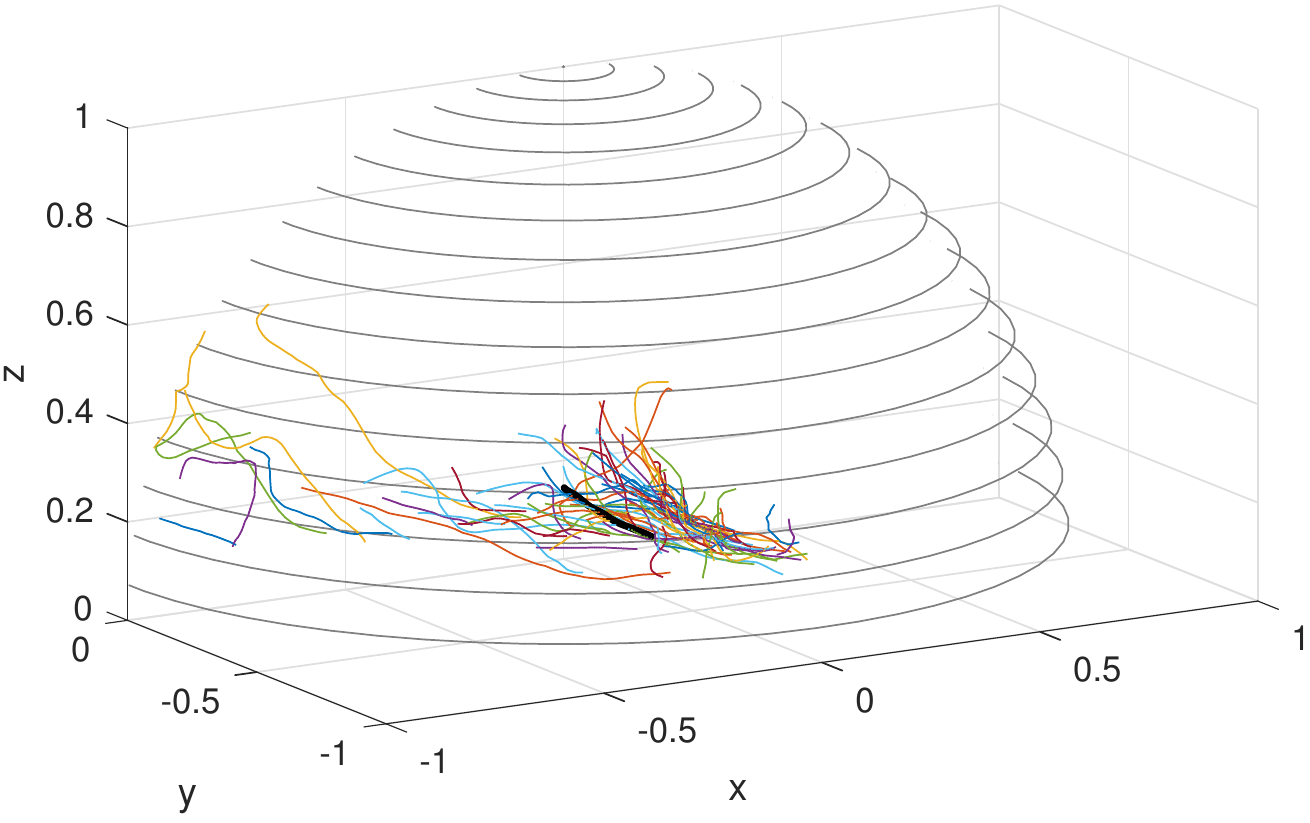}\hskip 0.5in
		\includegraphics[width=0.233\linewidth]{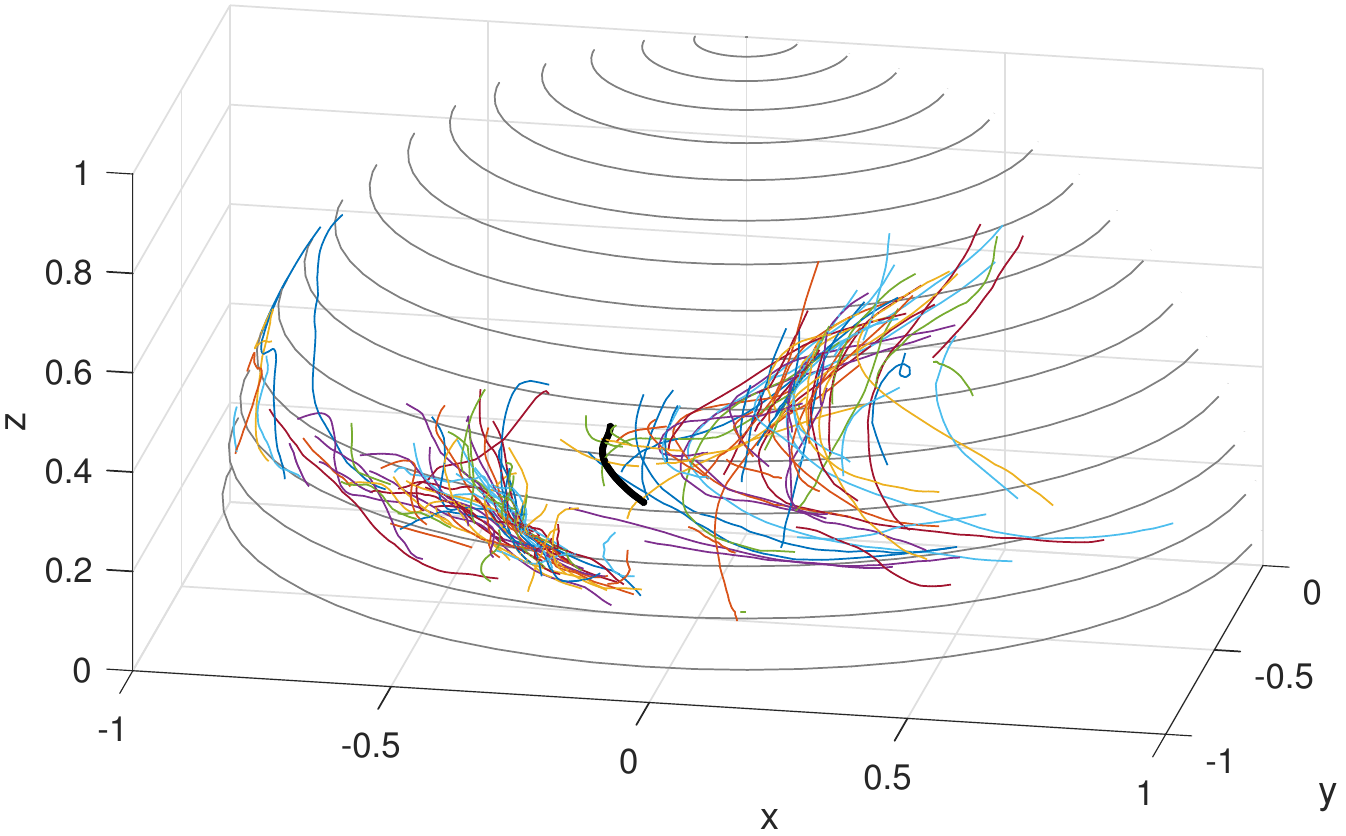}\\
		
		\includegraphics[width=0.233\linewidth]{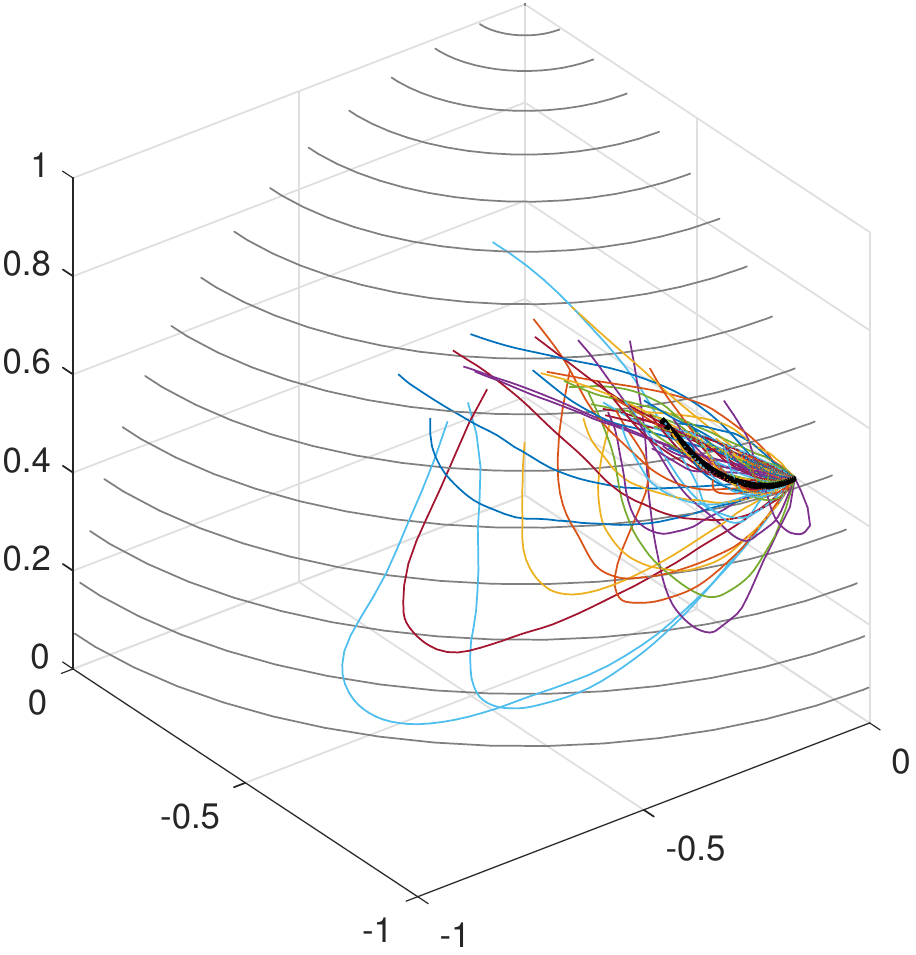}\hskip 0.5in
		\includegraphics[width=0.233\linewidth]{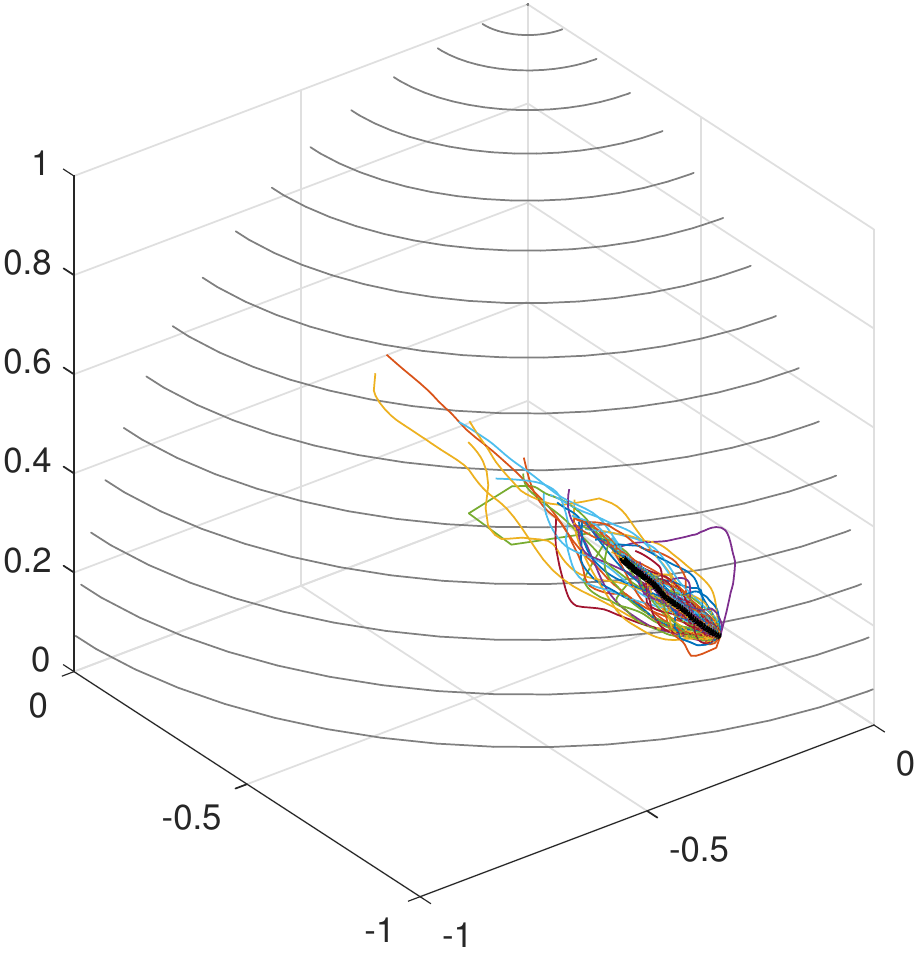}\hskip 0.5in
		\includegraphics[width=0.233\linewidth]{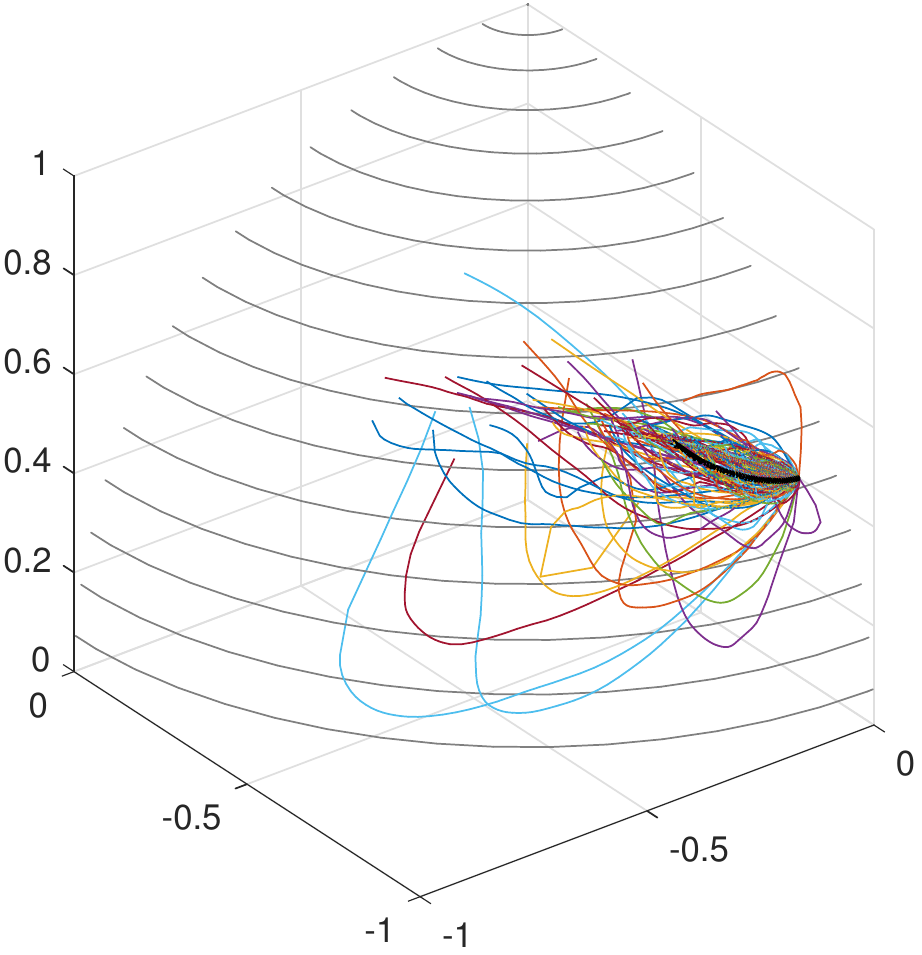}
	\end{center}
	\caption{In $\mathcal{S}^{AC}([0,1], S^2)$ (first row) and in $\mathcal{S}^{AC}([0,1], S^2)/\operatorname{SO}(3)$ (second row): the Karcher mean (black) of 75 hurricane tracks in Atlantic database, of 75 hurricane tracks in North Central Pacific hurricane database and of the combined 150 hurricane tracks (from left to right).}
	\label{KarcherMeans}
\end{figure*}
Let $\beta\in AC([0,1],S^n)$. Then a lift $\alpha\in AC([0,1], \operatorname{SO}(n+1))$ of $\beta$ can be computed as follows:
\begin{enumerate}[(1)]
	\item If $\beta(0)=-{\bf n}$, set 
	\begin{align}
		\alpha(0)=\left(  \begin{array}{lcr}
		-1 & 0 & 0\\0 & I_{n-1} & 0\\0 & 0 & -1
		\end{array} \right),
	\end{align} where $I_{n-1}$ is the $(n-1)\times(n-1)$ identity matrix. For
	$\beta(0)\neq-{\bf n}$, set $\alpha(0)=R_{{\bf n},\beta(0)}$.
	\item Given $\alpha(t)$, set $\alpha(t+\Delta t)=R_{\beta(t), \beta(t+\Delta t)}\alpha(t)$ for a chosen step size $\Delta t$.
\end{enumerate}

\begin{proposition}
	Given $\beta\in AC([0,1],S^n)$, by using the lift algorithm described above, the lift $\alpha$ is in  $AC^{\perp}([0,1], \operatorname{SO}(n+1))$, that is, it satisfies the two properties:
	\begin{enumerate}
		\item  $\beta(t)=\pi(\alpha(t))$ for all $t\in I$,
		\item  $\dot{\alpha}(t)\perp \alpha(t)T_I\operatorname{SO}(n)$ for all $t\in I$.
	\end{enumerate}
\end{proposition}

\begin{proof}
Obviously, the first property holds. And the discrete form of second also holds, that is, the geodesic between $\alpha(t)$ and $\alpha(t+\Delta t)$ is perpendicular to the orbits with respect to these two points. Assume that $\alpha(t+\Delta t)=y\alpha(t)$ for $y\in \operatorname{SO}(n+1)$. By the bi-invariance of the metric, we have the distance $d(\alpha(t), \alpha(t+\Delta t))=d(\alpha(t), y\alpha(t))=d(I, y).$ It is easy to see that $y$ left translates the orbit $\alpha(t)$ to the orbit $\alpha(t+\Delta t)$, which is equivalent to left translating $\beta(t)$ to $\beta(t+\Delta t)$, that is, $y\beta(t)=\beta(t+\Delta t)$. $R_{\beta(t),\beta(t+\Delta t)}\in \operatorname{SO}(n+1)$ is the most efficient rotation such that $d(\alpha(t), R_{\beta(t),\beta(t+\Delta t)}\alpha(t))=d(I, R_{\beta(t),\beta(t+\Delta t)})$ is smallest, which means the distance between $\alpha(t)$ and $R_{\beta(t),\beta(t+\Delta t)}\alpha(t)$ realizes the shortest possible distance between all pairs of representatives of these two orbits. 
\end{proof}

\section{Applications to hurricane tracks}
Finally we want to demonstrate the effectiveness of the proposed framework using real data.
We consider 75 hurricane tracks from the Atlantic hurricane database and 75 hurricane tracks from the  Northeast and North Central Pacific hurricane database (HURDAT2)\footnote{The data was obtained from the National Hurricane Center website: \url{http://www.nhc.noaa.gov/data/}.}. The data under consideration is depicted in Fig.~\ref{KarcherMeans}. Each hurricane path is represented as a curve in $S^2$, and is discretized as a piecewise-geodesic polygon. Our first step is to calculate the matrix of all pairwise distances. For unparametrized curves using an Intel Core i7-4510U (2.00GHz) machine, the computation of these 11175 boundary value problems took 
less than two minutes \footnote{In the implementation we made use of the one-dimensionality of $K=SO(2)$, which allowed us to solve the minimization over $K$ without the gradient method.}. 
We note that our algorithms are orders of magnitudes faster than the algorithms developed in
\cite{ZhSuKlLeSr2015,LeArBa2015,Lebrigant2016}, while at the same time overcoming the disadvantages of the methods used in \cite{SuKuKlSr2014}. 

\begin{figure}[htbp]
	\begin{center}
		\includegraphics[width=0.49\linewidth]{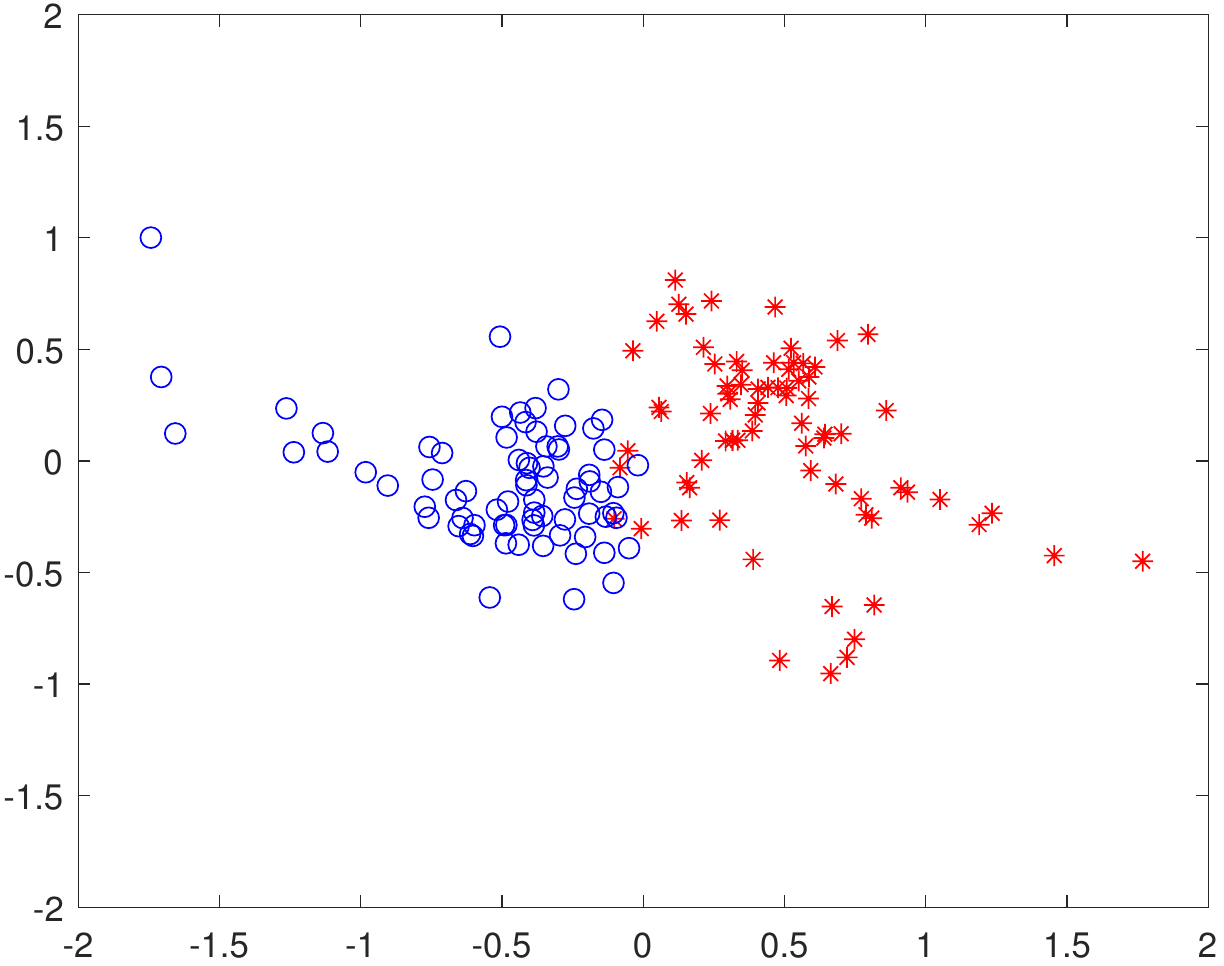}
		\includegraphics[width=0.49\linewidth]{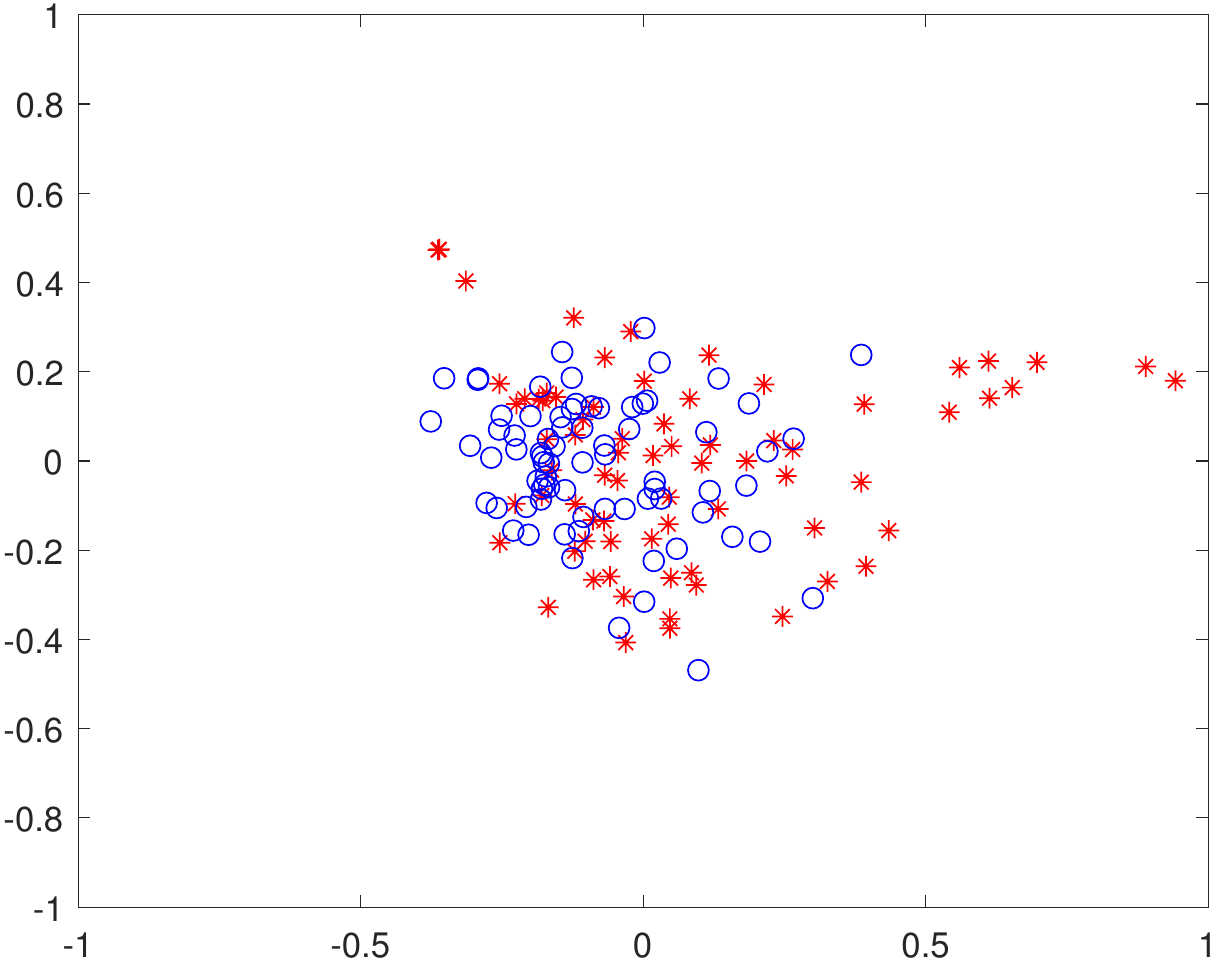}
	\end{center}
	\caption{The distance matrix of 150 hurricane tracks visualized using multi-dimensional scaling in two dimensions. Left: distances calculated in $\mathcal{S}^{AC}([0,1],S^2)$. Right: distances calculated in $\mathcal{S}^{AC}([0,1],S^2)/SO(3)$. Data points representing hurricanes from the Atlantic are marked with a $\star$; hurricanes from the Pacific region with a $\circ$.}
	\label{Clustering}
\end{figure}

In Figure~\ref{Clustering}
we visualized the distance matrices using multi-dimensional scaling \cite{KrWi1978}. 
As one might expect there is clear clustering between the hurricane tracks from the Atlantic region and those of the Northeast and North Central Pacific region if we regard them as elements of
$\mathcal{S}^{AC}([0,1],S^2)$. However if, in addition, we mod out by 
rigid motions, then the obtained distance matrix does not seem to capture this information anymore. This suggests that the clustering in the previous experiment was mainly based on location and that the shape of a hurricane path does not possess enough information to allow for a significant statement on its region of origin. 
Finally we calculate the Karcher mean of all hurricanes and of each of the groups separately as well. 
These results are depicted  in Fig.~\ref{KarcherMeans}. Using the Karcher means as a charts, this potentially allows to linearize the shape space using the corresponding  tangent spaces. As an example, we  show geodesics from the Karcher mean in the direction of the first two principal directions in Fig.~\ref{PrincipalDirection}. It seems that the first principal direction encodes the variety in shape, whereas the second direction seems to mainly reflect the change in the length of the hurricanes.




\begin{figure}[htbp]
	\begin{center}
		\includegraphics[width=0.4\linewidth]{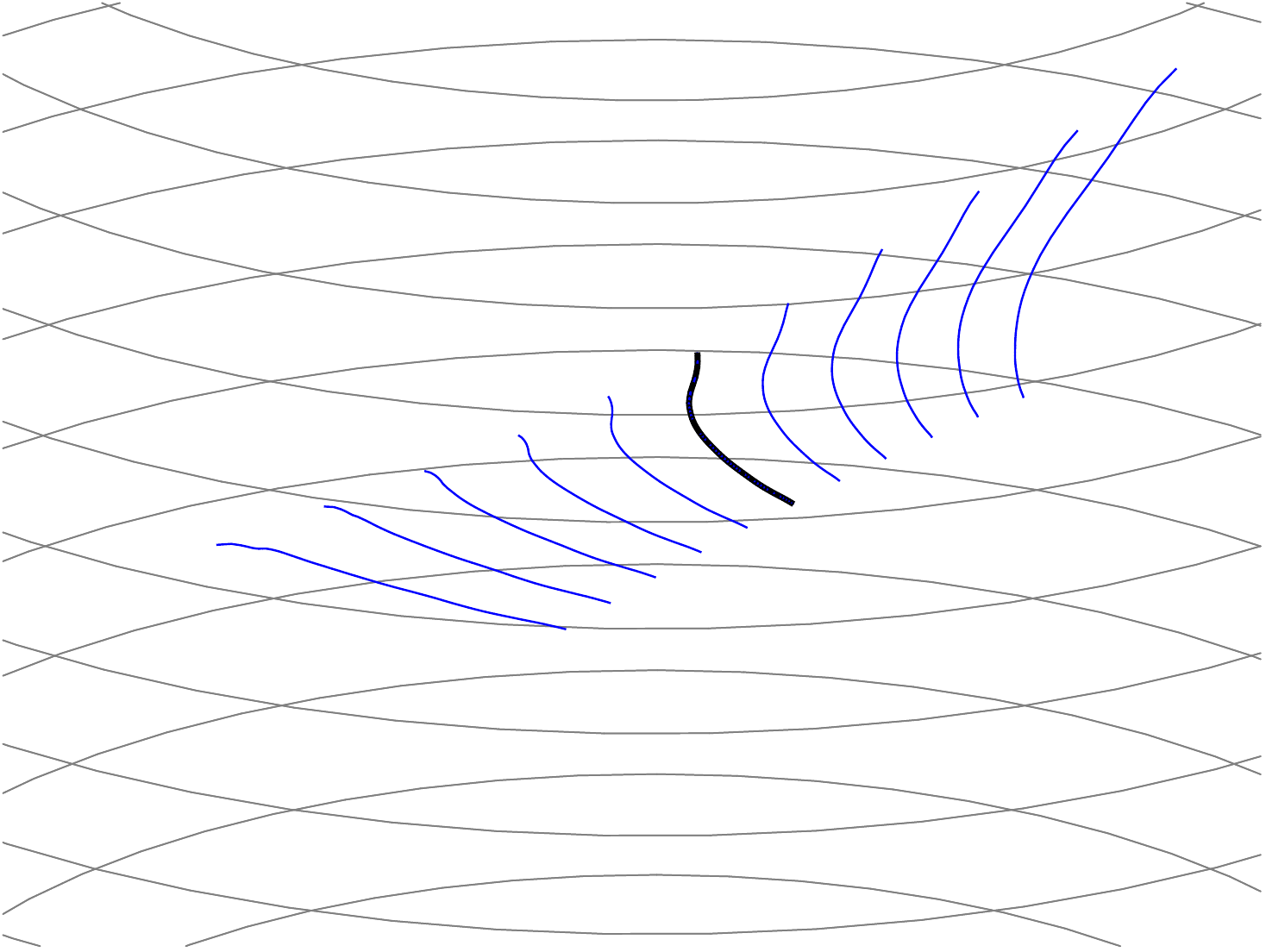}
		\includegraphics[width=0.4\linewidth]{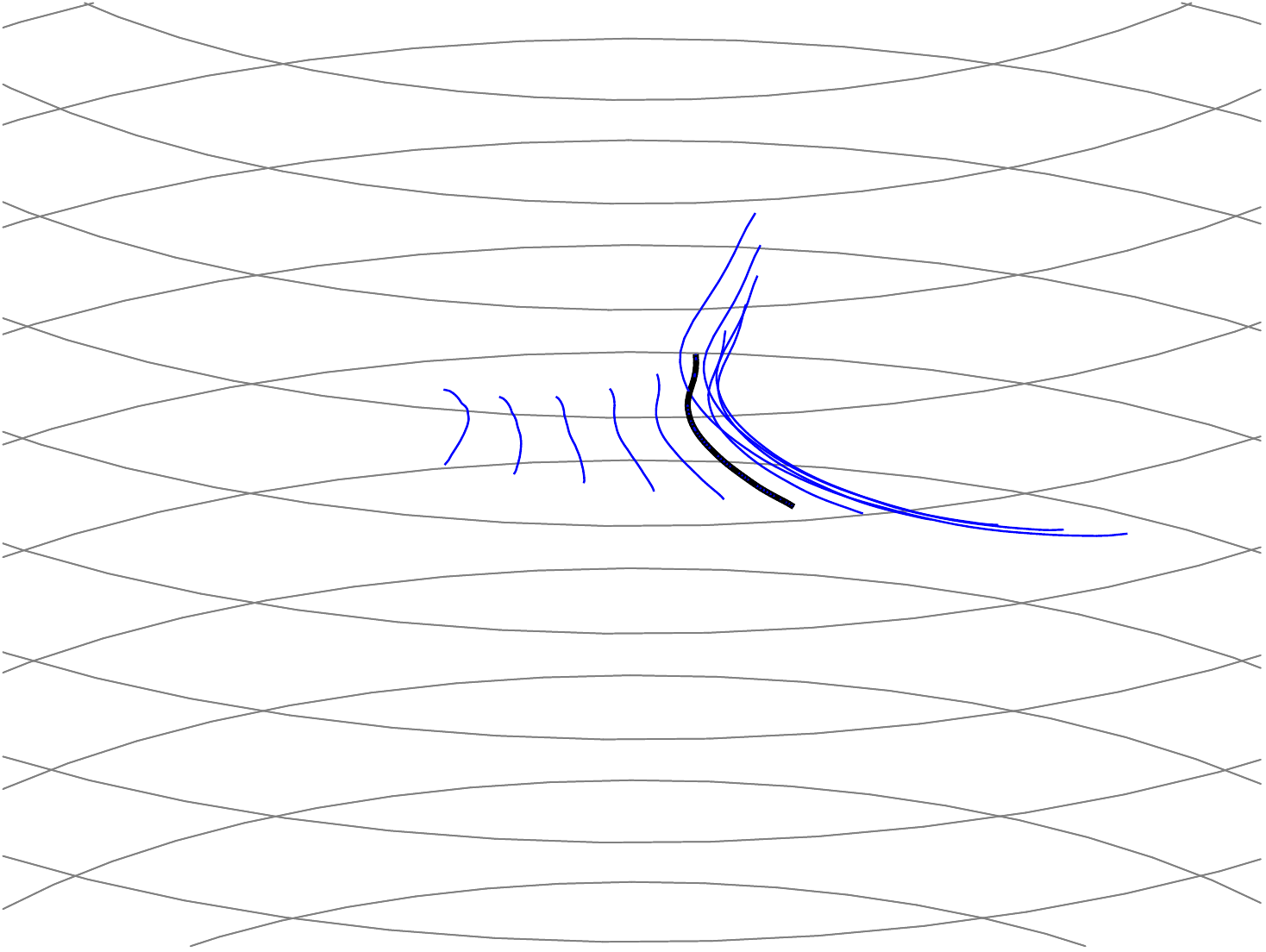}
		\\
		\includegraphics[width=0.4\linewidth]{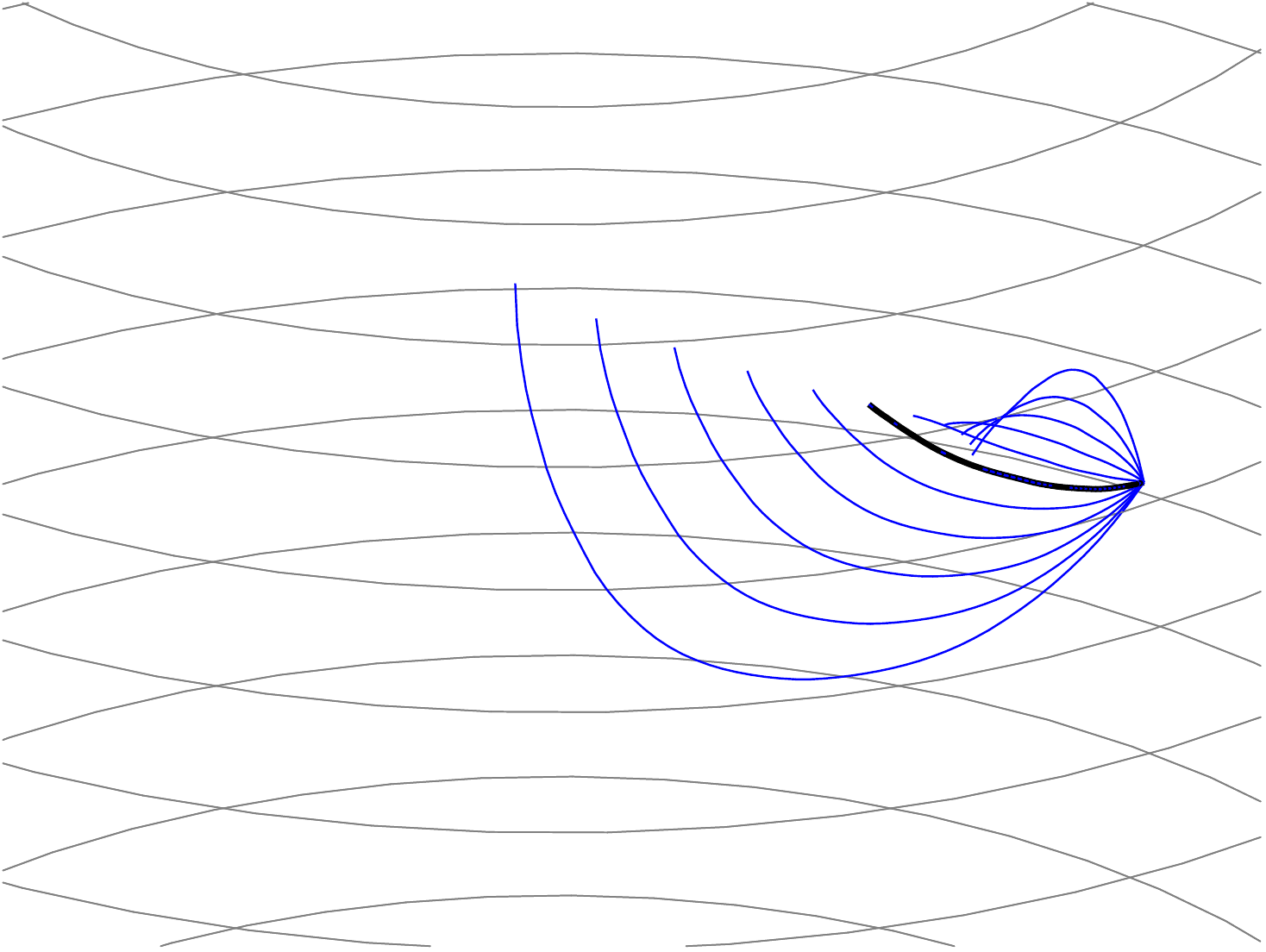}
		\includegraphics[width=0.4\linewidth]{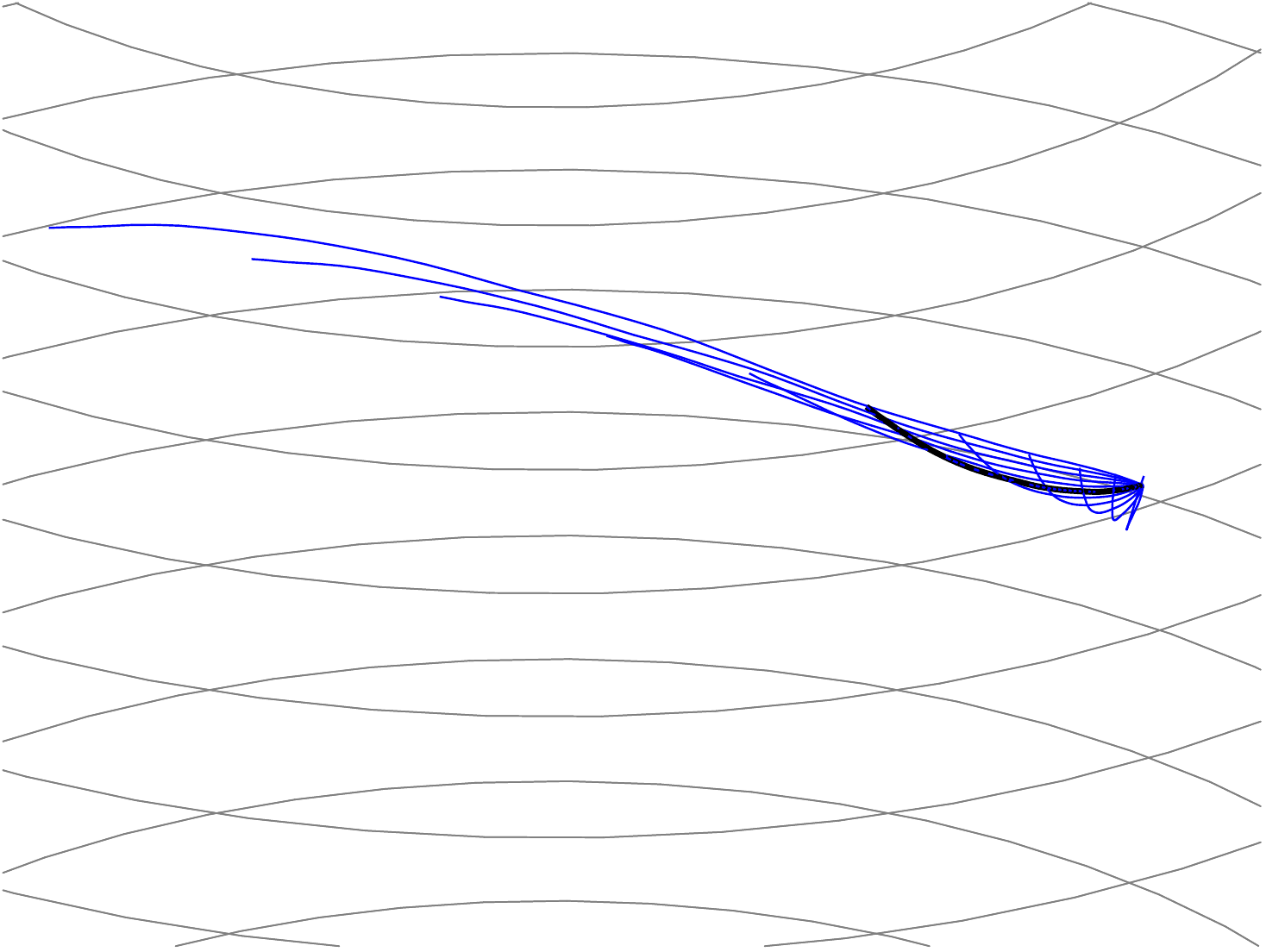}
	\end{center}
	\caption{The first two principal directions starting from the Karcher mean (black) of 150 hurricane tracks in $\mathcal{S}^{AC}([0,1],S^2)$ (first row) and in $\mathcal{S}^{AC}([0,1],S^2)/\operatorname{SO}(3)$ (second row). }
	\label{PrincipalDirection}
\end{figure}

{\small
	\bibliographystyle{ieee}
	\bibliography{egbib}
}

\end{document}